\documentclass[a4paper, 11pt, reqno]{article}
\usepackage{preamble}
\newcommand*{\keywords}[1]{\par\addvspace\medskipamount{\rightskip=0pt plus1cm
\def\and{\ifhmode\unskip\nobreak\fi\ $\cdot$
}\noindent \textbf{Keywords} \enspace\ignorespaces#1\par}}

\usepackage[
    style=numeric,
    giveninits=true,
    sortcites,
  ]{biblatex}

\title{On existence of measure with given marginals supported on a hyperplane\thanks{The article was prepared within the framework of the HSE University Basic Research Program.}}

\author{Alexander~P.~Zimin\thanks{National Research University Higher School of Economics, Russian Federation \& Center for Advanced Studies, Skoltech, Moscow, Russian Federation}}
\date{}

\begin{document}

%\begin{abstract}
%For a finite collection of absolutely continuous probability measures on the real line with compact supports and nonincreasing density functions we construct a transport plan concentrated on the hyperplane $\{x_1 + \dots + x_N = C\}$ with respective marginals. This transport plan is an optimal solution of the multi-marginal Monge-Kantorovich problem for the repulsive harmonic cost function $\sum_{i, j = 1}^N-(x_i - x_j)^2$.
%\end{abstract}
\maketitle
\begin{abstract}
Let $\{\mu_k\}_{k = 1}^N$ be absolutely continuous probability measures on the real line such that every measure $\mu_k$ is supported on the segment $[l_k, r_k]$ and the density function of $\mu_k$ is nonincreasing on that segment for all $k$. We prove that if $\mathbb{E}(\mu_1) + \dots + \mathbb{E}(\mu_N) = C$ and if $r_k - l_k \le C - (l_1 + \dots + l_N)$ for all $k$, then there exists a transport plan with given marginals supported on the hyperplane $\{x_1 + \dots + x_N = C\}$. This transport plan is an optimal solution of the multimarginal Monge-Kantorovich problem for the repulsive harmonic cost function $\sum_{i, j = 1}^N-(x_i - x_j)^2$.
\end{abstract}

\keywords{Optimal transportation \and Monge-Kantorovich problem \and Multimarginal transportation problem \and Repulsive harmonic interaction}

\section{Introduction}
\subsection{Multi-marginal OT with repulsive harmonic cost}

Assume we are given $N$ Polish spaces $X_1$, $X_2$, \dots, $X_N$, equipped with probability measures $\mu_k$ on $X_k$ and a cost function $c: X_1 \times \dots \times X_N \to \mathbb{R}$. In the multimarginal Monge-Kantorovich problem we seek to minimize $$\int_{X_1 \times \dots \times X_N}
c(x_1, x_2, \dots, x_N)~\gamma(dx_1, dx_2, \dots, dx_N)$$
over the set $\Pi(\mu_1, \mu_2, \dots, \mu_N)$ of positive measures $\gamma$ on the product space
$X_1 \times \dots \times X_N$ whose marginals are the $\mu_k$. See \cite{Villani, bogachev} for an account of the optimal transportation problem with two marginals and \cite{brendan_pass}.

An interesting example of the optimal transportation problem was studied in connection with applications to the density functional theory, namely, the Hohenberg–Kohn theory. The Hohenberg–Kohn theory considers a model of $N$ electrons whose arrangement in the space $\R^{3N}$ is determined by the density $\rho_{N}(x_1, \dots, x_N)$. The energy of pairwise interaction of electrons is specified as the density integral over the Coulomb potential:
$$
\mathcal{V}_{ee} = \int\sum_{1 \le i < j \le N} \frac{\rho_N(x_1, \dots, x_N)}{|x_i-x_j|} \, dx_1 \cdots dx_N.
$$
Due to the symmetry,
$$
\mathcal{V}_{ee} = \int \sum_{1 \le i < j \le N} \frac{\rho_2(x, y)}{|x-y|} \, dxdy,
$$
where $\rho_2(x, y) = \int {\rho_N}(x, y, z_2,  \cdots, z_N) \, dz_2 \cdots dz_N$.

In the Hohenberg-Kohn theory, the ground state is described by a functional that depends only on the density of one electron
\[\rho(x) = N \int {\rho_N}(x, z_2, \cdots, z_N) \, dz_2 \cdots dz_N.\]
For this purpose $\mathcal{V}_{ee}(\rho_2)$ is approximated by the functional ${\mathcal V}_{ee}(\rho)$, depending only on $\rho$.
The correct approximation is the key problem in this theory.

It turns out that the natural approximation is the approximation by the functional
$$
{F}(\rho) = \inf_{\pi \in \Pi(\rho,\rho)} \int \frac{1}{|x-y|} \, \pi(dx, dy).
$$
For example, this functional occurs when the so-called "semi-classical limit" is taken. Trivially, the functional $F$ is the Kantorovich functional (for the pair of equal marginals) with the cost function  $\frac{1}{|x - y|}$. This cost function is called Coulomb cost function.

In \cite{Cotar} the passage to the limit is made rigorously, and some sufficient conditions for the existence and uniqueness of a solution for the Kantorovich functional are found. For a generalization to a wider class of "repulsive cost functions" see \cite{CoDPMa}. For further progress in physical applications, see \cite{BindPasc}. In \cite{ChafaiMaida} transport inequalities and concentration inequalities for the Coulomb cost function are obtained.

In this paper we consider $X_k = \R^d$ for all $k = 1, \dots, N$ and the repulsive cost function $c$ having the form 
\[
c(x_1, \dots, x_N) = -\sum_{i, j = 1}^N|x_i - x_j|^2.
\]
This function describes repulsive harmonic electron-electron interaction. From a technical point of view, harmonic cost function is an interesting toy model to approach the case of Coulomb cost. From applications, the model with repulsive harmonic interaction allows particles to overlap, which makes it difficult to apply this model in practice. In \cite{dim-ger-nen} the properties of multi-marginal OT problem with repulsive harmonic cost are obtained and examples of solutions are given. 

Suppose that all marginals $\mu_k$ have  finite second moments. First, we notice that minimizers of this problem are also minimizers of the problem with the cost $c(x_1, \dots, x_N) = |x_1 + \dots + x_N|^2$: indeed, we have 
\begin{align*}
\int\sum_{i, j = 1}^N-|x_i - x_j|^2\gamma(dx_1, \dots, dx_N) = 
2\int c\,d\gamma 
- 2N\sum_{k = 1}^N\int|x_k|^2\mu_k(dx_k)
\end{align*}
and the last part does not depend on $\gamma$. Hence, if all marginals $\mu_k$ have  finite second moments, then we get
\[
\argmin_{\gamma}\int\sum_{i, j = 1}^N-|x_i - x_j|^2\gamma(dx_1, \dots, dx_N) = \argmin_{\gamma}\int|x_1 + \dots + x_N|^2\gamma(dx_1, \dots, dx_N).
\]

Clearly, $\argmin |x_1 + \dots + x_N|^2$ is the hyperplane $x_1 + \dots + x_N = 0$. Thus, if $\gamma \in \Pi(\mu_1, \dots, \mu_N)$ and $\supp(\gamma) \subset \{x_1 + \dots + x_N = 0\}$, then the transport plan $\gamma$ is trivially optimal. The following statement generalizes this observation. 
\begin{lemma}[{{\cite[Lemma 4.3]{dim-ger-nen}}}]
Let $\{\mu_k\}_{k = 1}^N$ be probability measures on $\R^d$ and $h\colon \R^d \to \R$ be a strictly convex function and suppose $c \colon (\R^d)^N \to \R$ be a cost function of the form $c(x_1, \dots, x_N) = h(x_1 + \dots + x_N)$. Then if there exists a plan $\gamma \in \Pi(x_1, \dots, x_N)$ concentrated on the hyperplane of the form $x_1 + \dots + x_N = C$, this plan is optimal for the multi-marginal problem with cost $c$. Moreover, if such $\gamma$ exists, then $\supp(\widehat{\gamma}) \subset \{x_1 + \dots + x_N = C\}$ is a necessary and sufficient condition for $\widehat{\gamma}$ to be optimal. In this case we will say that $\gamma$ is a flat optimal plan and $\{\mu_k\}_{k = 1}^N$ is a flat $N$-tuple of measures.
\end{lemma}

This reveals a very large class of minimizers in some cases.  However, not every $N$-tuple of measures is flat (see \cite[Remarks 4.4 and 4.5]{dim-ger-nen}). The next theorem provides an example of a flat $N$-tuple of measures.
\begin{theorem}[{{\cite[Theorem 4.6]{dim-ger-nen}}}]
For $k = 1, \dots, N$ let $\mu_k = \mu = \mathcal{L}^d|_{[0, 1]^d}$ be the uniform measure on $d$-dimensional cube $[0, 1]^d \subset \R^d$. Let $h \colon \R^d \to \R$ be a convex function and suppose that $c \colon ([0, 1]^d)^N \to \R$ is a cost function such that $c(x_1, \dots, c_N) = h(x_1 + \dots + x_N)$. Then, there exists a  transport map $T: [0, 1]^d \to [0, 1]^d$ such that $T_*(\mu) = \mu$, $T^N(x) = x$ and
\[
\min_{\gamma \in \Pi(\mu_1, \dots, \mu_N)}\int c(x_1, \dots, x_N)~\gamma(dx_1, \dots, dx_N) = \int c(x, T(x), \dots, T^{N - 1}(x))~\mu(dx).
\]
\end{theorem}
It follows from the proof of this theorem that $T$ does not depend on $h$ and $x + T(x) + \dots + T^{N - 1}(x) = \left(\frac{N}{2}, \dots, \frac{N}{2}\right)$ for all $x \in [0, 1]^d$. Then if we denote by $F$ the mapping $x \mapsto (x, T(x), \dots, T^{N - 1}(x))$, the measure $\gamma = F_*([0, 1]^d)$ is a flat optimal plan. This shows that a tuple of $N$ uniform measures on $d$-dimensional cube $[0, 1]^d$ is flat. See also \cite[Examples 4.9 and 4.10]{dim-ger-nen} for more examples of flat $N$-tuples of measures.

In this paper we present a wide class of flat $N$-tuples of measures for the case $d = 1$. Suppose that each measure of the tuple is absolutely continuous, concentrated on the segment and the density function is nonincreasing on this segment. Denote by $\mathbb{E}(\mu) = \int x\,\mu(dx)$ the first moment of the measure $\mu$. The following theorem gives necessary and sufficient conditions when $N$-tuple of measures with that properties is flat.
\begin{theorem}\label{thm:main_theorem}
Let $\{\mu_k\}_{k = 1}^N$ be absolutely continuous probability measures on the real line. Suppose that $\supp(\mu_k) = [l_k, r_k]$ and the density function of $\mu_k$ is nonincreasing on the segment $[l_k, r_k]$ for all $k = 1, \dots, N$. Then the $N$-tuple $\{\mu_k\}_{k = 1}^N$ is flat if and only if $r_k - l_k \le \mathbb{E}(\mu_1) + \dots + \mathbb{E}(\mu_N) - (l_1 + \dots + l_N)$ for all $k = 1, \dots, N$.
\end{theorem}
The necessity of these inequalities is trivial, and the proof of the sufficiency is the main part of this paper. In \cref{sec:D_lr_description} we consider the convex set $\mathcal{D}_{AC}[l, r]$ of probability measures concentrated on the segment $[l, r]$ with nonincreasing density functions and its closure $\mathcal{D}[l, r] \supset \mathcal{D}_{AC}[l, r]$ in the weak topology.
After that we consider the subset $\mathcal{D}[l, r; e] \subset \mathcal{D}[l, r]$ of measures with nonincreaing densities and with the fixed first moment equal to $e$. The set $\mathcal{D}[l, r; e]$ is a closed convex set, and we prove (see \cref{cor:ex_points_of_D_lre}) that every extreme point of $\mathcal{D}[l, r; e]$ has the form $\alpha \lambda[l, p^{(1)}] + (1 - \alpha)\lambda[l, p^{(2)}]$, where $l \le p^{(1)} \le p^{(2)} \le r$ and $\lambda[a, b]$ is the normalized restriction of the Lebesgue measure to the segment $[a, b]$.

In \cref{sec:VNC} we construct the set $\mathcal{V}^N[C]$ of $N$-tuples of probability measures $\{\mu_k\}_{k = 1}^N$ satisfying the conditions of \cref{thm:main_theorem} with the fixed constant $C$. In that section we directly extend the concept of extreme points for nonconvex sets and call them \textit{subextreme points}. First we construct the convex set $\mathcal{D}^N[\vec{l}, \vec{r}; \vec{e}\,] = \prod_{k = 1}^N \mathcal{D}[l_k, r_k; e_k]$. The set $\mathcal{V}^N[C]$ is a union of sets $\mathcal{D}^N[\vec{l}, \vec{r}; \vec{e}\,]$ with different parameters $\vec{l}$, $\vec{r}$ and $\vec{e}$, and therefore every subextreme point of $\mathcal{V}^N[C]$ must be an extreme point of one of the sets $\mathcal{D}^N[\vec{l}, \vec{r}; \vec{e}\,]$. This means (see \cref{prop:2_ex_of_V_C_is_step_C_tuple}) that each item of a subextreme tuple has the form $\alpha \lambda[l, p^{(1)}] + (1 - \alpha)\lambda[l, p^{(2)}]$. Next, we explicitly represent each tuple of this form as a nontrivial convex combination of some elements from $\mathcal{V}^N[C]$ except for the case when each item of the tuple is a Dirac measure. In particular (see \cref{thm:2_ex_of_V_C}), we show that every subextreme point of $\mathcal{V}^N[C]$ has the form $\{\delta(l_k)\}_{k = 1}^N$ with $l_1 + \dots + l_N = C$. 

In \cref{sec:final} we consider the convex closure of $\mathcal{V}^N[C]$. Using a converse version of the Krein-Milman theorem we prove that every extreme point of the compact restriction of $\overline{\mathrm{co}}\,\mathcal{V}^N[C]$ is a subextreme point of $\mathcal{V}^N[C]$, and therefore has the form $\{\delta(l_k)\}_{k = 1}^N$ with $l_1 + \dots + l_N = C$. For every such tuple there exists a trivial transport plan with the required projections concentrated on the hyperplane $\{x_1 + \dots + x_N = C\}$. Finally, using the Krein-Milman theorem we prove that every point of $\mathcal{V}^N$ is a flat $N$-tuple of probability measures since every extreme point of the convex closure of $\mathcal{V}^N[C]$ is flat.

\subsection{Connection with other results}
The problem of the existence of a probability measure with fixed marginals and a given support was previously studied in several works. In 
\cite{strassen1965, Kellerer} authors proved the following criterion for the case of two spaces.
\begin{theorem*}
Let $X$ and $Y$ be Polish spaces equipped with probability measures $\mu$ and $\nu$ respectively, and let $B$ be a closed subset of $X \times Y$. Then there exists a probability measure $\pi$ on $X \times Y$ with marginals $\mu$ and $\nu$ with the property $\supp(\pi) \subset B$ if and only if $\mu(B_1) + \nu(B_2) \ge 1$ for every couple of closed sets $B_1 \subset X$ and $B_2 \subset Y$ such that $B \subset \mathrm{Pr}^{-1}_X(B_1) \cup \mathrm{Pr}^{-1}_Y(B_2)$.
\end{theorem*}
\noindent Unfortunately, this theorem cannot be generalized for the case of multimarginal problem.

The existence of a transport plan concentrated on the hyperplane is a particular case of a Monge-Kantorovich problem with additional linear constraints is considered in the paper \cite{Zaev}. For given subspace $W$ of an appropriate functional space on $X_1 \times \dots \times X_N$ we consider the following optimization problem:
\[
\inf\left\{\int_{X_1 \times \dots \times X_N} c\,d\gamma \colon \gamma \in \Pi(\mu_1, \dots, \mu_N)\text{ and } \int \omega\,d\gamma = 0 \text{ for all } \omega \in W\right\}.
\]

Let $\omega$ be a continuous function such that $\omega > 0$ for all points not contained in the hyperplane $\{x_1 + \dots + x_N = C\}$ and that $\omega = 0$ otherwise. Then $\int \omega\,d\gamma = 0$ if and only if the transport plan $\gamma$ is concentrated on that hyperplane, so we can consider the space $W = \langle \omega \rangle$ in the transportation problem with additional constraints. Using the results from \cite{Zaev} one can easily verify that the tuple of measures $\{\mu_k\}_{k = 1}^N$ is flat if and only if for every collection of bounded continuous functions $\{\varphi_k\}_{k = 1}^N$, $\varphi_k \colon X_k \to \mathbb{R}$ satisfying the inequality 
\begin{equation}\label{eq:dual_part_1}
\varphi_1(x_1) + \dots + \varphi_N(x_N) \ge -\omega(x_1, \dots, x_N)
\end{equation}
for all $x_k \in X_k$ we have
\begin{equation}\label{eq:dual_part_2}
\int_{X_1}\varphi(x_1)\,\mu(dx_1) + \dots + \int_{X_N}\varphi_N(x_N)\,\mu(dx_N) \ge 0.
\end{equation}

One can easily prove that the tuple of measures $\{\mu_k\}_{k = 1}^N$ is flat if and only if there exist random variables $\{\xi_k\}_{k = 1}^N$ such that $\mathrm{law}(\xi_k) = \mu_k$ and
\begin{equation}\label{eq:martingale_like}
\mathbb{E}(\xi_k \mid \xi_1, \dots, \xi_{k - 1}, \xi_{k + 1}, \dots, \xi_N) = C + \xi_k - (\xi_1 + \dots + \xi_N)
\end{equation}
for all $k = 1, \dots, N$ and for an appropriate constant $C$. Indeed, in this case
\begin{align*}
\mathbb{E}(\xi_k(\xi_1 + \dots + \xi_N - C)) &= \mathbb{E}(\mathbb{E}(\xi_k(\xi_1 + \dots + \xi_N - C) \mid \xi_1, \dots, \xi_k, \xi_{k + 2}, \dots, \xi_N))\\
&= \mathbb{E}(\xi_k \cdot \mathbb{E}(\xi_1 + \dots + \xi_N - C \mid \xi_1, \dots, \xi_k, \xi_{k + 2}, \dots, \xi_N)) \\
&= \mathbb{E}(\xi_k \cdot 0) = 0,
\end{align*}
and therefore
\begin{align*}
    \mathbb{E}((\xi_1 + \dots + \xi_N - C)^2) = \sum_{k = 1}^N\mathbb{E}(\xi_k(\xi_1 + \dots + \xi_N - C)) - C \cdot \mathbb{E}(\xi_1 + \dots + \xi_N - C) = 0.
\end{align*}
This means that $\xi_1 + \dots + \xi_N = C$ with probability 1 and the joint distribution $\mathrm{law}(\xi_1, \dots, \xi_N)$ is a desired transport plan.

The problem of finding random variables satisfying equation \cref{eq:martingale_like} looks similar to the martingale transportation problem. In this problem for given probability distributions $\mu$ and $\nu$ and for given cost function $c$ we need to minimize
$\mathbb{E}(c(X, Y))$
over all couples of random variables $(X, Y)$ such that $\mathrm{law}(X) = \mu$, $\mathrm{law}(Y) = \nu$ and $\mathbb{E}(Y \mid X) = X$. More information on the martingale problem can be found in \cite{beiglbock2016}. Strassen \cite{strassen1965} showed that the existence of such couple $(X, Y)$ is equivalent to the fact that the measures $\mu$ and $\nu$ are in convex order, i.e. $\int \varphi\,d\mu \le \int \varphi\,d\nu$ for any convex function $\varphi$. We propose that there is a similar criterion under which the tuple of probability measures $\{\mu_k\}_{k = 1}^N$ is flat. 

Another connected problem is a multistochastic Monge-Kantorovich problem. In this problem for a given integer number $1 \le k < N$ we consider the $(N, k)$ minimization problem
$\int c d \pi \to \inf$ of the space of measures with fixed projections onto every  $X_{i_1} \times \dots \times X_{i_k}$
for arbitrary set of $k$ indices $\{i_1, \dots, i_k\} \subset \{1, \dots, N\}$. In \cite{GlKoZi} authors consider the multistochastic Monge-Kantorovich problem for the spaces $X_i = [0, 1]$, $1 \le i \le 3$ with the fixed 2-dimensional Lebesgue measures on the square $[0, 1]^2$ and with the cost function $\pm x_1x_2x_3$. The cost function $-x_1x_2x_3$ in a sense generalizes the function $-x_1x_2 - x_1x_3 - x_2x_3$, and an optimal transport plan in this problem is concentrated on the set $x_1 \oplus x_2 \oplus x_3 = 1$, where $\oplus$ is a bitwise addition. In addition, in \cite{multistochastic2020} authors show that the $(N, k)$-multistochastic problem is equivalent to the multimarginal transportation problem with $\binom{N}{k}$ spaces and additional constraint that the transport plan is concentrated on the lower-dimensional hyperplane, which is a natural generalization of our problem.

The general marginal problem was also studied in \cite{Kellerer64}. In this problem we need to find a transport plan with the fixed multidimensional marginals. Unlike the multimarginal problem, the existence of such transport plan is not guaranteed. In \cite{Kellerer64} author proved the dual criterion for a transport plan existence, which can be generalized to our problem as in equations \cref{eq:dual_part_1,,eq:dual_part_2}.

\Cref{thm:main_theorem} is also connected with multi-marginal optimal transportation with the product cost function $c(x_1, \dots, x_N) = x_1x_2 \dots x_N$. In the paper \cite{product_cost} authors consider the multi-marginal Monge-Kantorovich problem for the $3$-tuple of measures $(\mu_1, \mu_2, \mu_3)$, where all $\mu_i$ are the uniform measures on the segment $[0, 1]$ and the cost function $c(x_1, x_2, x_2) = x_1x_2x_3$. After the change of coordinates $y_k = -\ln(x_k)$ the product cost problem transforms into the following: we need to construct a minimizer of the functional
\[
\int_{\R^3} e^{-(y_1 + y_2 + y_3)}\gamma(dy_1, dy_2, dy_3)
\]
over all transport plans $\gamma$ such that $\prj{k}{\gamma} = \nu_k = e^{-y_k}\cdot \mathbbm{1}[y_k \ge 0]\,dy_k$ for all $k \in \{1, 2, 3\}$.

Here, most of the conditions of \cref{thm:main_theorem} are satisfied: for $1 \le k \le 3$, the measure $\nu_k$ is absolutely continuous, $\supp(
\nu_k) = [0, +\infty)$ and the density function of $\nu_k$ is decreasing on $[0, +\infty)$; in addition, the cost function $c(y_1, y_2, y_3) = h(y_1 + y_2 + y_3)$, where $h(x) = e^{-x}$ is convex. Unfortunately, we can not apply \cref{thm:main_theorem} since $\supp(\mu_k)$ is not a segment. Moreover, it was proved in \cite{product_cost} that an optimal solution $\gamma$ of the problem can be written as $\gamma = \gamma_1 + \gamma_2$, where $\gamma_1$ is one-dimensional and $\gamma_2$ is concentrated on the flat triangle (and hence on the plane) $\{(x_1, x_2, x_3) \colon l \le x_k \le r \text{ and } x_1 + x_2 + x_3 = 2l + r\}$ and $\prj{k}{\gamma_2} = \nu_k|_{[l, r]}$. However, \cref{thm:main_theorem} is applicable to the restrictions of $\nu_k$ to the segment $[l, r]$, which was our main motivation to prove it.

\section{Measures with nonincreasing density function}\label{sec:D_lr_description}
\subsection{Notation and preliminaries from the measure theory}
Let $X$ be a Polish space. We denote by $\mathcal{M}(X)$ the space of all signed Borel measures of bounded variation on $X$. We equip the space $\mathcal{M}(X)$ with the weak topology $\sigma(\mathcal{M}(X), C_b(X))$.

The space $\mathcal{M}(X)$ with the weak topology is a locally convex vector space. Let $\mathcal{M}^+(X) \subset \mathcal{M}(X)$ be a subspace of nonnegative measures equipped with the induced weak topology. Instead of $\mathcal{M}(X)$, the space $\mathcal{M}^+(X)$ is metrizable, and therefore is first-countable (see \cite[Theorem 8.9.4]{Bogachev_measure_vol2}). In particular, the subset $A \subset \mathcal{M}^+(X)$ is closed if and only if for every sequence $\{\mu^{(n)}\}_{n = 1}^{\infty}$ of nonnegative measures contained in $A$ such that $\{\mu^{(n)}\}$ converges weakly to $\mu$ we have $\mu \in A$.

Suppose that $K$ is homeomorphically embedded into a Polish space $X$ and the image of $K$ is closed in $X$. For every measure $\mu \in \mathcal{M}(K)$ let $\widehat{\mu}$ denote its extension to $\mathcal{B}(X)$ defined by $\widehat{\mu}(B) = \mu(B \cap K)$. The mapping $\mu \mapsto \widehat{\mu}$ is a homeomorphic embedding of $\mathcal{M}(K)$ into $\mathcal{M}(X)$. Denote by $\mathcal{M}^+(K; X)$ the subspace of nonnegative measures $\mu \in \mathcal{M}^+(X)$ such that $\supp(\mu) \subset K$. Then the image of $\mathcal{M}^+(K)$ under the mapping $\mu \mapsto \widehat{\mu}$ coincides with $\mathcal{M}^+(K; X)$. In the following paper we will identify the space $\mathcal{M}(K)$ with the subspace of $\mathcal{M}(X)$, and the space $\mathcal{M}^+(K)$ with $\mathcal{M}^+(K; X)$. 

Denote by $\mathcal{P}(X) \subset \mathcal{M}^+(X)$ the space of probability measures on $X$. Let $K$ be a compact subset of $X$. Then $\mathcal{P}(K) \subset \mathcal{P}(X)$, and it follows from the Prokhorov theorem (see \cite[Theorem 8.6.2]{Bogachev_measure_vol2}) that $\mathcal{P}(K)$ is sequentially compact, and therefore, since $\mathcal{P}(K) \subset \mathcal{M}^+(K)$ and $\mathcal{M}^+(K)$ is metrizable, the space $\mathcal{P}(K)$ is a compact convex subset of $\mathcal{M}(K)$.

Let us recall the following common properties of the weak convergence.

\begin{theorem}[{{\cite[Corollary 8.2.10]{Bogachev_measure_vol2}}}]
Let $\{\mu^{(n)}\}_{n = 1}^{\infty}$ be a sequence of probability measures, and let $\mu$ be a probability measure on a Polish space $X$. Suppose that the sequence $\{\mu^{(n)}\}_{n = 1}^\infty$ converges weakly to $\mu$. Then
\begin{assertions}
    \item \label{st:weak_measure_of_closed_set} for every closed set $F$ we have
    $$
    \limsup_{n \to \infty} \mu^{(n)}(F) \le \mu(F);
    $$
    \item \label{st:weak_measure_of_open_set} for every open set $U$ we have
    $$
    \liminf_{n \to \infty} \mu^{(n)}(U) \ge \mu(U).
    $$
\end{assertions}
\end{theorem}

\begin{corollary}\label{cor:support_of_converging_measures}
If $K$ is a closed subset of $X$, then $\mathcal{P}(K)$ is a closed subset of $\mathcal{P}(X)$.
\end{corollary}

In the following part of the paper, in most of the cases $X$ is a real line $\R$, and $K$ is a closed segment. So let us recall the following criterion of the weak convergence of measures on the real line in terms of cumulative distribution functions.
\begin{corollary}[{{\cite[Corollary 8.2.11]{Bogachev_measure_vol2}}}]\label{cor:convergence_dist_functions}
A sequence $\{\mu^{(n)}\}_{n = 1}^{\infty}$ of probability measures on the real line converges weakly to a probability measure $\mu$ precisely when the corresponding distribution functions $F_{\mu}^{(n)}$ converge to the distribution function $F_{\mu}$ of the measure $\mu$ at the points of continuity of $F_\mu$, where $F_\mu(t) = \mu((-\infty, t))$.
\end{corollary}

\subsection{Measures with nonincreasing density functions}

\begin{definition}
Given a segment $[l, r]$ on the real line. Denote \[\lambda[l, r] = \frac{1}{r - l}\mathcal{L}|_{[l, r]},\] where $\mathcal{L}$ is the Lebesgue measure  on the real line. In addition, denote by $\lambda[l, l]$ the Dirac measure concentrated at the point $l$. Thus, the measure $\lambda[l, r]$ is well-defined and $\lambda[l, r] \in \mathcal{P}(\R)$ for all $l \le r$
\end{definition}

\begin{definition}
Given a segment $[l, r]$ on the real line, denote by $\mathcal{D}_{AC}[l, r]$ the subset of probability measures $\mu \in \mathcal{P}[l, r]$ such that $\mu = \rho(x)\,dx$ for some nonincreasing density function $\rho \in L^1[l, r]$. We allow the function $\rho$ to be unbounded: in this case $\rho(l) = +\infty$.
\end{definition}

It trivially follows from the definition that $\mathcal{D}_{AC}[l, r]$ is a convex subset of $\mathcal{P}[l, r]$. For any $t \in (l, r]$ the measure $\lambda[l, t]$ is contained in $\mathcal{D}_{AC}[l, r]$. If the sequence $\{t^{(n)}\}_{n = 1}^{\infty} \subset (l, r]$ converges to $l$, then the sequence $\{\lambda[l, t^{(n)}]\}_{n = 1}^{\infty}$ converges weakly to the measure $\lambda[l, l] = \delta(l)$. Since $\delta(l)$ is singular, this measure is not contained in $\mathcal{D}_{AC}[l, r]$; hence, the space $\mathcal{D}_{AC}[l, r]$ is not closed. Let us construct the closure of this space.

\begin{definition}
Given a segment $[l, r]$ on the real line, denote by $\mathcal{D}[l, r]$ the subset of measures $\mu \in \mathcal{P}[l, r]$ that can be written as \[\mu = \alpha\delta(l) + (1 - \alpha)\mu',\] where $0 \le \alpha \le 1$ and $\mu' \in \mathcal{D}_{AC}[l, r]$. In addition, we denote by $\mathcal{D}[l, l]$ the space containing the unique measure $\delta(l)$.
\end{definition}

We claim that the closure of $\mathcal{D}_{AC}[l, r]$ coincides with $\mathcal{D}[l, r]$. To prove this, we give a description of $\mathcal{D}[l, r]$ in terms of a cumulative distribution function.

\begin{proposition} \label{prop:D_description_dist_function}
Given a probability measure $\mu$ on the real line, denote by $F_\mu$ the cumulative distribution function of $\mu$: $F_\mu(t) = \mu((-\infty, t))$. Then the space $\mathcal{D}[l, r]$ contains the measure $\mu$ if and only if
\begin{conditions}
    \item \label{cond:supp_mu_on_lr} $F_\mu(l) = 0$ and $F_\mu(r+) = 1$;
    \item \label{cond:F_mu_is_concave} $F_\mu$ is concave on $(l, +\infty)$.
\end{conditions}
\end{proposition}
\begin{proof}
Suppose that $\mu \in \mathcal{D}[l, r]$. Since $\supp(\mu) \subset [l, r]$, we have $F_\mu(l) = \mu((-\infty, l)) = 0$ and $F_\mu(r+) = \mu((-\infty, r]) = 1$, and therefore \cref{cond:supp_mu_on_lr} holds.

By definition, $\mu = \alpha \delta(l) + (1 - \alpha)\mu'$, where $0 \le \alpha \le 1$ and $\mu' \in \mathcal{D}_{AC}[l, r]$. Since $\mu' \in \mathcal{D}_{AC}[l, r]$, there exists a function $\rho(x)$ such that $\rho(x)$ is nonincreasing on $[l, r]$ and is equal to $0$ outside that segment. If $t \in [l, +\infty)$, then \[F_{\mu'}(t) = \int_l^t\rho(x)\,dx,\]
and therefore, since $\rho(x)$ is nonincreasing on $[l, +\infty)$, the function $F_{\mu'}$ is concave on $[l, +\infty)$. Finally, $F_{\delta(l)}(x) = 1$ for all $x \in (l, +\infty)$. In particular, the cumulative distribution function $F_{\delta(l)}$ is concave on $(l, +\infty)$, and therefore $F_\mu = \alpha F_{\delta(l)} + (1 - \alpha)F_{\mu'}$ is also concave on $(l, +\infty)$. This implies \cref{cond:F_mu_is_concave}.

Suppose that $F_\mu$ satisfies \cref{cond:supp_mu_on_lr,,cond:F_mu_is_concave}. Since $F_\mu$ is concave on $(l, +\infty)$, this function is semi differentiable on $(l, +\infty)$. Denote by $\partial_+F_\mu(x)$ the right derivative  of $F_\mu$ at $x$ for all $x \in (l, +\infty)$. Since $F_\mu$ is concave on $(l, +\infty)$, the function $\partial_+F_\mu(x)$ is nonincreasing on that interval and
\[
\int_a^b \partial_+F_\mu(x)\,dx = F_\mu(b) - F_\mu(a)
\]
for all $[a, b] \subset (l, +\infty)$. See \cite[Proposition 1.6.1]{convex_functions} for a proof of this statement.

Since $F_\mu$ is nondecreasing on $(l, +\infty)$, we have $\partial_+F_\mu(x) \ge 0$ for all $x \in (l, +\infty)$. In addition, since $F_\mu$ is constant on $(r, +\infty)$, we have $\partial_+F_\mu(x) = 0$ for all $x > r$. Denote $\alpha = F_\mu(l+)$. We have $0 \le \alpha \le 1$. If $\alpha = 1$, then $\mu(\{l\}) = 1$, and therefore $\mu = \delta(l) \in \mathcal{D}[l, r]$. Otherwise, let
\begin{align*}
    \rho(x) = \begin{dcases}
        0 &\text{if $x < l$,}\\
        +\infty &\text{if $x = l$,}\\
        \frac{\partial_+F_\mu(x)}{1 - \alpha} &\text{if $x > l$.}
    \end{dcases}
\end{align*}
Let $\mu' = \rho(x)\,dx$. One can easily verify that $\mu' \in \mathcal{D}_{AC}[l, r]$, and the cumulative distribution function of the measure $\alpha\delta(l) + (1 - \alpha)\mu'$ is equal to $F_\mu$ at all points of the real line. Thus, $\mu = \alpha \delta(l) + (1 - \alpha)\mu'$, and therefore $\mu \in \mathcal{D}[l, r]$.
\end{proof}

\begin{corollary}\label{lem:mu_ab_bounded_from_above}
If $\mu \in \mathcal{D}[l, r]$, then \[
\mu([a, b]) \le \frac{b - a}{a - l}
\] for all $[a, b] \subset (l, +\infty)$.
\end{corollary}
\begin{proof}
Let $F_\mu$ be the cumulative distribution function of the measure $\mu$. It follows from \cref{cond:F_mu_is_concave} that $F_\mu$ is concave on $(l, +\infty)$. In particular, $F_\mu$ is continuous on this interval, and therefore $\mu([a, b]) = F_\mu(a) - F_\mu(b)$. By the concavity of $F_\mu$ we get
\[
\frac{F_\mu(b) - F_\mu(a)}{b - a} \le \frac{F_\mu(a) - F_\mu(p)}{a - p}
\]
for all $p \in (l, a)$. Since $F_\mu(a) - F_\mu(p) \le F_\mu(a) \le 1$, we have
\[
F_\mu(b) - F_\mu(a) \le \frac{b - a}{a - p}.
\]
Tending $p$ to $l$, we obtain the desired result.
\end{proof}

\begin{proposition}\label{prop:strong_closure_of_D}
Let  $K$ be a closed segment. Let $\{\mu^{(n)}\}_{n = 1}^{\infty}$ be a sequence of probability measures such that $\mu^{(n)} \in \mathcal{D}[l^{(n)}, r^{(n)}]$ and that $[l^{(n)}, r^{(n)}] \subset K$ for all $n$. Suppose that the sequence $\{\mu^{(n)}\}_{n = 1}^{\infty}$ converges weakly to a measure $\mu$. Then the sequence $\{l^{(n)}\}_{n = 1}^\infty$ is convergent and $\mu \in \Dec{l}{r}$, where $l = \lim\limits_{n \to \infty}l^{(n)}$ and $r = \liminf\limits_{n \to \infty}r^{(n)}$.
\end{proposition}

\begin{proof}
Let $F_\mu^{(n)}$ be the cumulative distribution function of the measure $\mu^{(n)}$, and let $F_\mu$ be a cumulative distribution function of $\mu$. Let $l = \liminf_{n \to \infty}$. Since $\supp(\mu^{(n)}) \subset K$ for all $n$, we conclude that $l^{(n)} \in K$ for all $n$, and therefore $l$ is finite.

Let $[a, b]$ be a segment contained in $(l, +\infty)$. Since $l = \liminf_{n \to \infty}l^{(n)}$, there exists a subsequence $\{\mu^{(k_n)}\} \subset \{\mu^{(n)}\}$ such that $l^{(k_n)} < (l + a) / 2$ for all $n$. It follows from \cref{lem:mu_ab_bounded_from_above} that \[
\mu^{(k_n)}([a, b]) 
\le \frac{b - a}{a - l^{(k_n)}} \le 2\frac{b - a}{a - l}\] for all $k_n$. Hence, using \cref{st:weak_measure_of_open_set} we conclude that
\[
F_\mu(b) - F_\mu(a+) = \mu((a, b)) \le \liminf_{n \to \infty} \mu^{(k_n)}((a, b)) \le \liminf_{n \to \infty}\mu^{(k_n)}([a, b]) \le  2\frac{b - a}{a - l}.
\]

Let $c$ be an arbitrary point on $(l, a)$. We have $F_\mu(c+) \le F_\mu(a)$, and therefore
\[
F_\mu(b) - F_\mu(a) \le F_\mu(b) - F_\mu(c+) \le 2\frac{b - c}{c - l}.
\]
Tending $c$ to $a$, we obtain
\begin{equation} \label{eq:F_mu_is_Lipschitz}
F_\mu(b) - F_\mu(a) \le 2\frac{b - a}{a - l}.
\end{equation}

Let $x$ and $y$ be arbitrary points on the segment $[a, b]$. Applying inequality \cref{eq:F_mu_is_Lipschitz} to the points $x$ and $y$, we conclude that
\[
|F_\mu(x) - F_\mu(y)| \le \frac{2}{\min(x, y) - l}|x - y| \le \frac{2}{a - l}|x - y|.
\]
Thus, $F_\mu$ is Lipschitz continuous on the segment $[a, b]$, and therefore is continuous at all points of $(a, b)$.

Taking into account \cref{cor:convergence_dist_functions}, we conclude that the sequence of functions $\{F_\mu^{(n)}\}$ converges pointwise to $F_\mu$ on the interval $(a, b)$. In particular, the subsequence $\{F_\mu^{(k_n)}\}$ converges pointwise to $F_\mu$ on the interval $(a, b)$, where $l^{(k_n)} < (l + a) / 2 < a$ for all $n$. By \cref{prop:D_description_dist_function}, the function $F_\mu^{(k_n)}$ is concave on $(a, b)$ for all $n$. Take into account the following statement.
\begin{lemma*}[{{\cite[Corollary 1.3.8]{convex_functions}}}]
If $f^{(n)} \colon I \to \R$ $(n \in \mathbb{N})$ is a pointwise converging sequence of convex functions, then its limit $f$ is also convex. 
\end{lemma*}
\noindent Using this lemma, we conclude that $F_\mu$ is concave on $(a, b)$ as the poinwise limit of concave function. Since $[a, b]$ is an arbitrary subsegment of $(l, +\infty)$, the function $F_\mu$ is continuous and concave on $(l, +\infty)$.

Let $l' = \limsup_{n \to \infty}l^{(n)}$. For any $\eps > 0$ there exists a subsequence $\{\mu^{(k_n)}\} \subset \{\mu^{(k)}\}$ such that $l^{(k_n)} > l' - \eps$ for all $n$. This means that $\supp(\mu^{(k_n)}) \subset [l' - \eps, +\infty)$ for all $n$. Using \cref{cor:support_of_converging_measures}, we conclude that $\supp(\mu) \subset [l' - \eps, +\infty)$ for all $\eps > 0$. Thus, $\supp(\mu) \subset [l', +\infty)$, and therefore $F_\mu(x) = 0$ for all $x \le l'$.

By construction, $l \le l'$. Suppose that $l < l'$. Then the function $F_\mu$ is concave on $(l, +\infty)$ and is equal to zero on $[l, l']$. Hence, $F_\mu(x) \le 0$ for all $x \ge l$, and this is impossible since $\mu$ is a probability measure.

This contradiction proves that \[\limsup_{n \to \infty}l^{(n)} = l' = l = 
\liminf_{n \to \infty}l^{(n)}.\]
Hence, the sequence $\{l^{(n)}\}$ is convergent and $\lim_{n \to \infty}l^{(n)} = l$. Furthermore, $F_\mu(x) = 0$ for all $x \le l$ and $F_\mu$ is concave on $(l, +\infty)$.

Finally, let $r = \liminf_{n \to \infty}r^{(n)}$. For any $\eps > 0$ there exists a subsequence $\{\mu^{(k_n)}\} \subset \{\mu^{(n)}\}$ such that $r^{(k_n)} < r + \eps$ for all $k_n$. This means that $\supp(\mu^{(k_n)}) \subset (-\infty, r + \eps]$, and therefore by \cref{cor:support_of_converging_measures} we have $\supp(\mu) \subset (-\infty, r + \eps]$ for all $\eps > 0$. Hence, $\supp(\mu) \subset (-\infty, r]$, and therefore $F_\mu(r+) = 1$. Thus, all the conditions of \cref{prop:D_description_dist_function} holds and we conclude that $\mu \in \mathcal{D}[l, r]$.
\end{proof}

\begin{corollary}
The set $\mathcal{D}[l, r]$ is a compact convex subset of $\mathcal{P}[l, r]$, and if $l < r$, then $\mathcal{D}[l, r]$ coincides with the closure of $\mathcal{D}_{AC}[l, r]$.
\end{corollary}
\begin{proof}
It trivially follows from \cref{prop:strong_closure_of_D} that $\mathcal{D}[l, r]$ is closed. Since $\mathcal{D}[l, r] \subset \mathcal{P}[l, r]$ and $\mathcal{P}[l, r]$ is compact, we conclude that $\mathcal{D}[l, r]$ is compact too. The convexity of $\mathcal{D}[l, r]$ trivially follows from the definition. 

Assume that $l < r$, and let $\overline{D}$ be the closure of $\mathcal{D}_{AC}[l, r]$. As we have already shown, the measure $\delta(l)$ is contained in $\overline{D}$. Since $\mathcal{D}_{AC}[l, r]$ is convex, the set $\overline{D}$ is also convex, and therefore the measure $\alpha \delta(l) + (1 - \alpha)\mu$ is contained in $\overline{D}$ for all $\mu \in \mathcal{D}[l, r]$ and for all $0 \le \alpha \le 1$. Thus, $\mathcal{D}[l, r] \subset \overline{D}$, and therefore, since $\mathcal{D}[l, r]$ is closed, we conclude that $\mathcal{D}[l, r] = \overline{D}$.
\end{proof}

\subsection{Extreme points of \texorpdfstring{$\mathcal{D}[l, r]$}{D[l, r]} and homeomorphism between \texorpdfstring{$\mathcal{D}[l, r]$}{D[l, r]} and \texorpdfstring{$\mathcal{P}[l, r]$}{P[l, r]}}
Let us recall to the reader the Krein-Milman theorem.
\begin{theorem*}[{{\cite{KM40}}}]
If $X$ a compact convex subset of a locally convex vector space, then $X$ is the closed convex hull of its extreme points. 
\end{theorem*}
Let us introduce the following common notation. If $A$ is a subset of a locally convex vector space, we denote by $\overline{\mathrm{co}}\,A$ the closed convex hull of $A$. In addition, if $A$ is convex, we denote by $\mathrm{ex}\,A$ the set of extreme points of $A$.

The set $\mathcal{D}[l, r]$ is a compact convex subset of the locally convex vector space $\mathcal{M}[l, r]$, and therefore it is fully described by its extreme points.

\begin{proposition}
The measure $\mu$ is an extreme point of $\mathcal{D}[l, r]$ if and only if $\mu = \lambda[l, t]$ for some $t \in [l, r]$.
\end{proposition}
\begin{proof}
If $l = r$, there is nothing to prove. So, we may assume that $l < r$.

Let $\mu$ be an extreme point of $\mathcal{D}[l, r]$. By definition, $\mu$ can be represented in the form $\alpha \delta(l) + (1 - \alpha)\mu'$, where $0 \le \alpha \le 1$ and $\mu' \in \mathcal{D}_{AC}[l, r]$. Since $\mu' \ne \delta(l)$ and $\mu \in \mathrm{ex}\,\mathcal{D}[l, r]$, we conclude that either $\mu = \delta(l) = \lambda[l, l]$ or $\mu \in \mathcal{D}_{AC}[l, r]$.

Assume that $\mu \in \mathcal{D}_{AC}[l, r]$. By definition, there exists a nonnegative nonincreasing density function $\rho$ defined on the segment $[l, r]$ such that $\mu = \rho\,dx$. Let $[l, s]$ be the support of $\mu$, and let $t$ be an arbitrary point on the interval $(l, s)$. Since $\supp(\mu) = [l, s]$, we conclude that $\rho(t) > 0$.

Consider the functions $\rho^{(1)}(x) = \min\{\rho(t), \rho(x)\}$ and $\rho^{(2)}(x) = \max\{\rho(x) - \rho(t), 0\}$. One can easily verify that both $\rho^{(1)}$ and $\rho^{(2)}$ are nonnegative nonincreasing functions defined on the segment $[l, r]$, and that the equation $\rho(x) = \rho^{(1)}(x) + \rho^{(2)}(x)$ holds for all $x \in [l, r]$.

Denote 
\[\alpha = \int_l^r\rho^{(1)}(x)\,dx = 1 - \int_l^r\rho^{(2)}(x)\,dx.\]
We have $0 \le \alpha \le 1$. Since $\rho^{(1)}(x) = \rho(t)$ for all $x \in [l, t]$, we conclude that $\alpha \ge \rho(t)(t - l) > 0$. If $\alpha = 1$, then $\rho(x) = \rho^{(1)}(x)$ for almost all $x \in [l, r]$, and in particular $\rho(x) = \rho(t)$ for almost all $x \in [l, t]$.

Suppose otherwise that $\alpha < 1$. Then $\mu = \alpha \mu^{(1)} + (1 - \alpha)\mu^{(2)}$, where $\mu^{(1)} = \rho^{(1)}(x)\,dx / \alpha$ and $\mu^{(2)} = \rho^{(2)}(x)\,dx / (1 - \alpha)$. Both measures $\mu^{(1)}$ and $\mu^{(2)}$ are contained in $\mathcal{D}_{AC}[l, r]$, and therefore, since $\mu \in \mathrm{ex}\,\mathcal{D}[l, r]$, we conclude that $\mu = \mu^{(1)} = \mu^{(2)}$. In particular, $\rho(x) = \rho^{(1)}(x) / \alpha$ for almost all $x \in [l, r]$, and therefore $\rho(x) = \rho(t) / \alpha$ for almost all $x \in [l, t]$.

In both cases we conclude that $\rho(x)$ is constant on $(l, t)$ for all $t \in (l, s)$. Hence, $\rho(x)$ is constant on $(l, s)$. In addition, since $\supp(\mu) = [l, s]$, we have $\rho(x) = 0$ for all $x > s$. Hence, $\mu$ is proportional to the restriction of the Lebesgue measure to the segment $[l, s]$, and therefore, since $\mu$ is a probability measure, we get $\mu = \lambda[l, s]$.

Thus, we conclude that if $\mu \in \mathrm{ex}\,\mathcal{D}[l, r]$, then $\mu = \lambda[l, t]$ for some $t \in [l, r]$. Let us prove the sufficiency of this condition. If $t = l$, then by definition $\lambda[l, l] = \delta(l) \in \mathcal{D}[l, r]$. Moreover, since $\delta(l)$ is an extreme point of $\mathcal{P}[l, r]$ and $\mathcal{D}[l, r] \subset \mathcal{P}[l, r]$, the measure $\delta(l)$ is also an extreme point of $\mathcal{D}[l, r]$.

We may assume that $l < t$. We have $\lambda[l, t] \in \mathcal{D}_{AC}[l, r] \subset \mathcal{D}[l, r]$. Suppose that $\lambda[l, t] = \alpha\mu^{(1)} + (1 - \alpha)\mu^{(2)}$ for some real number $\alpha \in (0, 1)$ and probability measures $\mu^{(1)}, \mu^{(2)} \in \mathcal{D}[l, r]$. The measure $\lambda[l, t]$ is absolutely continuous, and therefore both $\mu^{(1)}$ and $\mu^{(2)}$ are also absolutely continuous. Thus, $\mu^{(1)}, \mu^{(2)} \in \mathcal{D}_{AC}[l, r]$, and therefore $\mu^{(1)} = \rho^{(1)}\,dx$ and $\mu^{(2)} = \rho^{(2)}\,dx$ for some nonnegative nonincreasing functions $\rho^{(1)}$ and $\rho^{(2)}$. 

Since $\supp(\lambda[l, t]) = [l, t]$, we conclude that $\supp(\mu^{(1)}) \subset [l, t]$ and $\supp(\mu^{(2)}) \subset [l, t]$; in particular, $\rho^{(1)}(x) = \rho^{(2)}(x) = 0$ for all $x > t$. Since $\rho^{(1)}$ and $
\rho^{(2)}$ is nonincreasing on $[l, t]$ and their nontrivial convex combination $\alpha\rho^{(1)}(x) + (1 - \alpha)\rho^{(2)}(x)$ is equal to the constant function $1 / (t - l)$ for almost all $x \in [l, t]$, we conclude that both functions $\rho^{(1)}$ and $\rho^{(2)}$ are constant on $(l, t)$. Thus, both $\mu^{(1)}$ and $\mu^{(2)}$ are proportional to $\lambda[l, t]$, and therefore $\mu^{(1)} = \mu^{(2)} = \lambda[l, t]$. This proves that $\lambda[l, t] \in \mathrm{ex}\,\mathcal{D}[l, r]$.
\end{proof}

\begin{corollary} \label{cor:D_lr_is_closed_convex_hull}
$\mathcal{D}[l, r] = \overline{\mathrm{co}}\{\lambda[l, t] \colon l \le t \le r\}$.
\end{corollary}

We are going to present the natural linear homeomorphism between $\mathcal{P}[l, r]$ and $\mathcal{D}[l, r]$. By the Riesz-Markov-Kakutani representation theorem, the space $\mathcal{M}[l, r]$ is the topological dual of $C[l, r]$, and the weak topology on $\mathcal{M}[l, r]$ coincides with the weak$^*$ topology $\sigma(\mathcal{M}[l, r], C[l, r])$. In particular, any linear continuous operator $T \colon C[l, r] \to C[l, r]$ provides the dual linear operator $T^*\colon \mathcal{M}[l, r] \to \mathcal{M}[l, r]$, which is continuous in the weak topology.
\begin{lemma}\label{lem:T_operator_is_homeomorhic_embedding}
Consider the linear operator $T \colon C[l, r] \to C[l, r]$ defined as follows:
\begin{align*}
    T(\varphi)(t) = \begin{dcases}
    \frac{\int_l^t \varphi(x)\,dx}{t - l} &\text{if $l < t \le r$},\\
    \varphi(l) &\text{if $t = l$}.
    \end{dcases}
\end{align*}
Then $T$ is a continuous linear operator and the dual linear operator $T^*\colon \mathcal{M}[l, r] \to \mathcal{M}[l, r]$ provides a linear homeomorphic embedding from $\mathcal{P}[l, r]$ onto itself.
\end{lemma}
\begin{proof}
If $\varphi$ is a continuous function, then the function $T(\varphi)$ is also continuous. Indeed, it is enough to verify the continuity only at the point $t = l$. Since $\varphi$ is continuous at the point $l$, we have $\varphi(t) = \varphi(l) + o(1)$ as $t \to l$. Hence, $\int_l^t\varphi(x)\,dx = \varphi(l)(t - l) + o(t - l)$, and therefore $T(\varphi)(t) = \varphi(l) + o(1)$. Thus, $T(\varphi)(t) \to  T(\varphi)(l)$ as $t \to l$.

This means that $\mathrm{Im}(T) \subset C[l, r]$. In addition, we trivially have $\norm{T} \le 1$. Thus, the linear operator $T$ is bounded, and therefore is continuous.

Let us verify that $T^*$ is injective. Equivalently, we need to prove that $\mathrm{Im}(T)$ is dense in $C[l, r]$. Consider a function $\psi \in C^1[l, r]$ and denote
\[
\varphi(x) = \frac{d}{dx}((x - l) \cdot \psi(x)).
\]
We have $\varphi \in C[l, r]$. In addition, $\int_l^t \varphi(x)\,dx = (t - l)\cdot \psi(t)$, and therefore $T(\varphi)(t) = \psi(t)$ for all $t \in (l, r]$. Hence, $T(\varphi) = \psi$, and we conclude that $C^1[l, r] \subset \mathrm{Im}(T)$. Thus, $\mathrm{Im}(T)$ is dense in $C[l, r]$ and $T^*$ is injective.

If a function $\varphi$ is nonnegative, then $T(\varphi)$ is also nonnegative. Hence, if $\mu$ is a positive linear functional on $C[l, r]$, then $T^*(\mu)$ is also a positive linear functional, and therefore $T^*$ maps $\mathcal{M}^+[l, r]$ onto itself. Finally, since $T(1) = 1$, we conclude that the image of an arbitrary probability measure is also a probability measure. Thus, since $\mathcal{P}[l, r]$ is a compact set, it follows from the closed map lemma that $T^*$ is a linear homeomorphic embedding from $\mathcal{P}[l, r]$ onto itself.
\end{proof}
\begin{proposition}
The linear operator $T^*$ defined in \cref{lem:T_operator_is_homeomorhic_embedding} provides a homeomorphism between $\mathcal{P}[l, r]$ and $\mathcal{D}[l, r]$, which maps $\delta(t)$ to $\lambda[l, t]$ for all $t \in [l, r]$.
\end{proposition}
\begin{proof}
Since $\mathcal{P}[l, r]$ is a compact convex set and $T^*$ is a homeomorphic embedding, we have $T^*(\mathcal{P}[l, r])$ is also a compact convex set.

Any linear bijection between convex sets provides a one-to-one mapping between their extreme points. Thus, since $\mathrm{ex}\,\mathcal{P}[l, r] = \{\delta(t) \colon l \le t \le r\}$, we conclude that $\mathrm{ex}\,T^*(\mathcal{P}[l, r]) = \{T^*(\delta(t)) \colon l \le t \le r\}$.

For any continuous function $\varphi \in C[l, r]$ we have
\begin{align*}
\langle T^*(\delta(t)), \varphi \rangle = T(\varphi)(t) = \begin{dcases}
\int_l^t \varphi(x)\,\frac{dx}{t - l} &\text{if $l < t \le r$},\\
\varphi(l) &\text{otherwise}.
\end{dcases}
\end{align*}
Hence, $T^*(\delta(t)) = \lambda[l, t]$ for all $t \in [l, r]$, and therefore $\mathrm{ex}\,T^*(\mathcal{P}[l, r]) = \{\lambda[l, t] \colon l \le t \le r\}$. Since $T^*(\mathcal{P}[l, r])$ is a compact convex subset of the locally convex vector space $\mathcal{M}[l, r]$, by the Krein-Milman theorem we conclude that \[
T^*(\mathcal{P}[l, r]) = \overline{\mathrm{co}}\,\{\lambda[l, t] \colon l \le t \le r\}.
\]
Thus, it follows from \cref{cor:D_lr_is_closed_convex_hull} that $T^*(\mathcal{P}[l, r]) = \mathcal{D}[l, r]$.

Finally, since $T^*$ is a linear homeomorphic embedding, this operator provides a linear homeomorphism between $\mathcal{P}[l, r]$ and $T^*(\mathcal{P}[l, r]) = \mathcal{D}[l, r]$.
\end{proof}
As an application we describe extreme points of the set of measures with nonincreasing density function with the fixed first moment.

\begin{definition}
Given points $l$ and $r$ on the real line such that $l \le r$, and a real number $e$. We denote by $\mathcal{D}[l, r; e]$ the set of probability measures $\mu \in \mathcal{D}[l, r]$ such that $\mathbb{E}(\mu) = e$.
\end{definition}

The function $\mathbb{E} \colon \mathcal{P}[l, r] \to \R$ is continuous, and therefore the set $\mathcal{D}[l, r; e]$ is a closed subset of $\mathcal{D}[l, r]$; hence, since $\mathcal{D}[l, r]$ is compact, the set $\mathcal{D}[l, r; e]$ is also compact. In addition, since the function $\mathbb{E}$ is linear, the set $\mathcal{D}[l, r; e]$ is convex. Thus, the set $\mathcal{D}[l, r; e]$ is a compact convex subset of $\mathcal{M}[l, r]$, and therefore it is fully described by its extreme points.

Let $\mu \in \mathcal{P}[l, r]$. If we denote $
\mu^* = T^*(\mu) \in \mathcal{D}[l, r]$, then we have
\[
\int_l^r t\,\mu^*(dt) = \int_l^r \frac{\int_l^t x\,dx}{t - l}\,\mu(dt) = \int_l^r \frac{l + t}{2}\,\mu(dt).
\]
Thus, the linear operator $T^*$ provides a linear homeomorphism between the set \[\mathcal{P}[l, r; 2e - l] = \{\mu \in \mathcal{P}[l, r] \colon \mathbb{E}(\mu) = 2e - l\}\] and $\mathcal{D}[l, r; e]$, and therefore it provides a bidirectional mapping between $\mathrm{ex}\,\mathcal{P}[l, r; 2e - l]$ and $\mathrm{ex}\,\mathcal{D}[l, r; e]$.

We only need to find the extreme points of $\mathcal{P}[l, r; 2e - l]$. To do this, we use the following theorem describing extreme points of the set of probability measures having prescribed values for the integrals of (a finite number of) prescribed functions $f_i$. 

\begin{theorem}[{{\cite[Theorem 2.1]{karr_1983}}}]\label{thm:extreme_points_of_certain_sets}
Let $E$ be a compact metric space with the Borel $\sigma$-algebra. Let $f_1, 
\dots, f_n$ be bounded continuous functions on $E$, and let $c_1, \dots, c_n$ be real numbers. Denote by $K$ the set of Borel probability measures $
\mu$ for which
\[
\int_E f_i(x)\,\mu(dx) = c_i \text{ for $i = 1, \dots, n$}.
\]
Then for each $\mu \in K$ the following assertions are equivalent:
\begin{assertions}
    \item $\mu$ is an extreme point of $K$;
    \item $\#(\supp(\mu)) \le n + 1$ and if $\supp(\mu) = \{x_1, \dots, x_k\}$, then the vectors $v_i = (f_1(x_i), \dots, f_n(x_i), 1)$, $1 \le i \le k$, are linearly independent.
\end{assertions}
\end{theorem}
Using this theorem, we obtain the following result:
\begin{proposition}
The measure $\mu \in \mathcal{P}[l, r]$ is an extreme point of $\mathcal{P}[l, r; 2e - l]$ if and only if $\mathbb{E}(\mu) = 2e - l$ and $\mu$ can be written as
\begin{equation}\label{eq:ex_P_lre_representation}
\mu = \alpha \delta(p^{(1)}) + (1 - \alpha) \delta(p^{(2)}),
\end{equation}
where $l \le p^{(1)} \le p^{(2)} \le r$ and $0 < \alpha < 1$.
\end{proposition}
\begin{proof}
The segment $[l, r]$ is a compact metric space, and the function $f(x) = x$ is a bounded continuous function defined on this segment. In addition, if $k \le 2$ and the points $x_1, \dots, x_k$ are pairwise disjoint, then the vectors $(x_1, 1), \dots, (x_k, 1)$ are linearly independent. Thus, it follows from \cref{thm:extreme_points_of_certain_sets} that $\mu$ is an extreme point of $\mathcal{P}[l, r; 2e - l]$ if and only if $\mu \in \mathcal{P}[l, r; 2e - l]$ and $\#(\supp(\mu)) \le 2$.

One can easily verify that $\#(\supp(\mu)) \le 2$ if and and only if $\mu$ can be represented as in equation \cref{eq:ex_P_lre_representation}. Note that if $\#(\supp(\mu)) = 1$, then this representation is not unique.
\end{proof}
Finally, since $T^*$ provides a bidirectional mapping between $\mathrm{ex}\,\mathcal{P}[l, r; 2e - l]$ and $\mathrm{ex}\,\mathcal{D}[l, r; e]$ and $T^*(\delta(p)) = \lambda[l, p]$ for all $p \in [l, r]$, we obtain the following result.
\begin{corollary}\label{cor:ex_points_of_D_lre}
The measure $\mu$ is an extreme point of $\mathcal{D}[l, r; e]$ if and only if $\mathbb{E}(\mu) = e$ and $\mu$ can be written as
\[
\mu = \alpha \lambda[l, p^{(1)}] + (1 - \alpha)\lambda[l, p^{(2)}],
\]
where $l \le p^{(1)} \le p^{(2)} \le r$ and $0 < \alpha < 1$.
\end{corollary}

\section{The set \texorpdfstring{$\mathcal{V}^N[C]$}{VN[C]} and its subextreme points}\label{sec:VNC}

\subsection{Definition of the set \texorpdfstring{$\mathcal{V}^N[C]$}{VN[C]} and step \texorpdfstring{$C$}{C}-compatible tuples}
For a positive integer $N$, we denote by $\mathcal{M}^N(X)$  the space of $N$-tuples of measures $\vec{\mu} = (\mu_1, \dots, \mu_N)$ such that $\mu_k \in \mathcal{M}(X)$ for each $k = 1, \dots, N$. This space equipped with the product topology is a locally convex vector space. We also consider the subset $\mathcal{P}^N(X) \subset \mathcal{M}^N(X)$ of $N$-tuples of probability measures, and if $K$ is a compact subset of $X$, we will identify $\mathcal{P}^N(K)$ with a compact convex subset of $\mathcal{M}^N(X)$ as in the previous section.
\begin{definition}
Let $C$ be a real number. A pair of $N$-tuples of points $(\vec{l}, \vec{r})$ is called \textit{a $C$-compatible boundary} if the inequality
\[
0 \le r_k - l_k \le C - (l_1 + \dots + l_N)
\]
holds for all $k = 1, \dots, N$.
\end{definition}
\begin{definition}
Let $C$ be a real number. We say that an $N$-tuple of probability measures $\vec{\mu}$ is contained in the set $\mathcal{V}^N[C] \subset \mathcal{P}^N(\R)$ if $\mathbb{E}(\mu_1) + \dots + \mathbb{E}(\mu_n) = C$ and there exists a $C$-compatible boundary $(\vec{l}, \vec{r})$ such that $\mu_k \in \mathcal{D}[l_k, r_k]$ for all $k = 1, \dots, N$.
\end{definition}
To derive the sufficiency in \cref{thm:main_theorem}, it is enough to prove that if $\vec{\mu} \in \mathcal{V}^N[C]$, then $\vec{\mu}$ is a flat $N$-tuple of measures, or more precisely there exists a transport plan $\gamma$ concentrated on the hyperplane $\{x_1 + \dots + x_N = C\}$ such that $\mu_k = \prj{k}{\gamma}$ for all $k = 1, \dots, N$.

Let us find  sufficient conditions for an $N$-tuple of probability measures $\vec{\mu}$ with $\mu_k = \lambda[l_k, r_k]$ for each $k$ to belong $\mathcal{V}^N[C]$.
\begin{proposition}\label{prop:lebesgue_tuple_in_V_C_ineq}
For $1 \le k \le N$, let $\mu_k = \lambda[l_k, r_k]$. Suppose that
\[
C = \frac{l_1 + r_1}{2} + \dots + \frac{l_N + r_N}{2}
\]
and that the inequality $r_k - l_k \le \sum_{i \ne k}(r_i - l_i)$ holds for all $k = 1, \dots, N$. Then $\vec{\mu} \in \mathcal{V}^N[C]$.
\end{proposition}
\begin{proof}
Since $\mu_k = \lambda[l_k, r_k]$, we have $\mathbb{E}(\mu_k) = (l_k + r_k)/2$, and therefore
\[
\mathbb{E}(\mu_1) + \dots + \mathbb{E}(\mu_N) = \frac{l_1 + r_1}{2} + \dots + \frac{l_N + r_N}{2} = C.
\]

Let us verify that $(\vec{l}, \vec{r})$ is a $C$-compatible boundary. We have
\[
C - (l_1 + \dots + l_N) = \frac{r_1 - l_1}{2} + \dots + \frac{r_N - l_N}{2},
\]
and therefore the inequality $r_k - l_k \le C - (l_1 + \dots + l_N)$ is equivalent to $2(r_k - l_k) \le (r_1 - l_1) + \dots + (r_N - l_N)$, which trivially follows from the condition $r_k - l_k \le \sum_{i \ne k}(r_i - l_i)$. Thus, since $\mu_k \in \mathcal{D}[l_k, r_k]$ for all $k = 1, \dots, N$, we conclude that $\vec{\mu} \in \mathcal{V}^N[C]$.
\end{proof}

\begin{corollary}\label{cor:lebesgue_tuple_in_V_C_eq}
For $1 \le k \le N$, let $\mu_k = \lambda[l_k, r_k]$. Suppose that
\[
C = \frac{l_1 + r_1}{2} + \dots + \frac{l_N + r_N}{2}
\]
and that there exists an index $m = 1, \dots, N$ such that $r_m - l_m = \sum_{i \ne m}(r_i - l_i)$. Then $\vec{\mu} \in \mathcal{V}^N[C]$.
\end{corollary}
\begin{proof}
By \cref{prop:lebesgue_tuple_in_V_C_ineq} we only need to prove that the inequality $r_k - l_k \le \sum_{i \ne k}(r_i - l_i)$ holds for all $k = 1, \dots, N$. If $k = m$, then the equality is achieved. Since $r_m - l_m = \sum_{i \ne m}(r_i - l_i)$, we conclude that $r_m - l_m \ge r_k - l_k$ for all $k$. Hence, if $k \ne m$, then \[
\sum_{i \ne k}(r_i - l_i) \ge r_m - l_m \ge r_k - l_k,
\]
and therefore the inequality $r_k - l_k \le \sum_{i \ne k}(r_i - l_i)$ holds for all $k = 1, \dots, N$.
\end{proof}
The set $\mathcal{V}^N[C]$ is noncompact and even nonconvex. Nevertheless, in the this section we find "extreme" points of this set to prove in the following that each of them is a flat $N$-tuple of measures. 

If $X$ is a (not necessary convex) subset of a locally convex vector space, we say that a point $x$ is a \textit{subextreme point of the space $X$} if $x \in X$ and we cannot represent it in the form $x = \alpha x_1 + (1 - \alpha)x_2$, where $x_1$ and $x_2$ are distinct points of $X$ and $\alpha \in (0, 1)$. If $X$ is convex, then extreme and subextreme points of $X$ are the same, and therefore the definition of subextreme points is a continuation of the definition of extreme points on nonconvex sets. We will denote the set of subextreme points of $X$ by $\mathrm{se}\,X$.

Let us introduce the following definition.
\begin{definition}
An $N$-tuple of probability measures $\vec{\mu}$ is called \textit{a step $C$-compatible tuple} if $\mathbb{E}(\mu_1) + \dots + \mathbb{E}(\mu_N) = C$ and $\mu_k$ can be written as \[\mu_k = \alpha_k\lambda[l_k, p_k] + (1 - \alpha_k)\lambda[l_k, r_k]\] for each $k = 1, \dots, N$, where $(\vec{l}, \vec{r})$ is a $C$-compatible boundary, $l_k \le p_k \le r_k$ and $0 < \alpha_k < 1$ for all $k$.
\end{definition}

One can easily verify that every step $C$-compatible tuple is contained in $\mathcal{V}^N[C]$. If $\vec{\mu}$ is a step $C$-compatible tuple, then $\supp(\mu_k) = [l_k, r_k]$, which follows from the inequality $\alpha_k < 1$. Hence, the $N$-tuples of points $\vec{l}$ and $\vec{r}$ are uniquely defined. In addition, since $\alpha > 0$, the point $p_k$ is also uniquely defined by $\mu_k$ for all $k$. In the following we will use the same notation for the $N$-tuples of points $\vec{l}$, $\vec{r}$ and $\vec{p}$ in the context of step $C$-compatible tuples.

The following statement motivates the definition of a step $C$-compatible tuple.

\begin{proposition} \label{prop:2_ex_of_V_C_is_step_C_tuple}
If $\vec{\mu} \in \mathrm{se}\,\mathcal{V}^N[C]$, then $\vec{\mu}$ is a step $C$-compatible tuple.
\end{proposition}
\begin{proof}
Since $\vec{\mu} \in \mathcal{V}^N[C]$, by construction $\mathbb{E}(\mu_1) + \dots + \mathbb{E}(\mu_N) = C$ and there exists a $C$-compatible boundary $(\vec{l}, \vec{r}^{\,(0)})$ such that $\mu_k \in \mathcal{D}[l_k, r_k^{(0)}]$ for all $k$.

For $1 \le k \le N$, denote $e_k = \mathbb{E}(\mu_k)$. Consider the set \[\mathcal{D}^N[\vec{l}, \vec{r}^{\,(0)}; \vec{e}\,] = \mathcal{D}[l_1, r_1^{(0)}; e_1] \times \dots \times \mathcal{D}[l_N, r_N^{(0)}; e_N] \subset \mathcal{M}^N(\R).\]
The set $\mathcal{D}^N[\vec{l}, \vec{r}^{\,(0)}; \vec{e}\,]$ is a compact convex subset of $\mathcal{M}^N(\R)$ as a product of compact convex sets. In addition, it follows from the definition of $\mathcal{V}^N[C]$ that $\mathcal{D}^N[\vec{l}, \vec{r}^{\,(0)}; \vec{e}\,] \subset \mathcal{V}^N[C]$. Hence, since $\vec{\mu} \in \mathcal{D}^N[\vec{l}, \vec{r}^{\,(0)}; \vec{e}\,]$ and $\vec{\mu}$ is a subextreme point of $\mathcal{V}^N[C]$, we conclude that $\vec{\mu}$ is also an extreme point of $\mathcal{D}^N[\vec{l}, \vec{r}^{\,(0)}; \vec{e}\,]$.

The $N$-tuple of measures $\vec{\mu}$ is an extreme point of $\mathcal{D}^N[\vec{l}, \vec{r}^{\,(0)}; \vec{e}\,]$ if and only if $\mu_k \in \mathrm{ex}\,\mathcal{D}[l_k, r_k^{(0)}; e_k]$ for all $k = 1, \dots, N$. Thus, by \cref{cor:ex_points_of_D_lre} for each $k = 1, \dots, N$ there exist points $p_k^{(1)}$ and $p_k^{(2)}$ and a real number $\alpha_k \in (0, 1)$ such that $l_k \le p_k^{(1)} \le p_k^{(2)} \le r_k^{(0)}$ and $\mu_k = \alpha_k\lambda[l_k, p_k^{(1)}] + (1 - \alpha_k)\lambda[l_k, p_k^{(2)}]$.

Put $r_k = p_k^{(2)}$ and $p_k = p_k^{(1)}$. Since $l_k \le r_k \le r_k^{(0)}$ for all $k$ and $(\vec{l}, \vec{r}^{\,(0)})$ is a $C$-compatible boundary, we can easily verify that $(\vec{l}, \vec{r})$ is also a $C$-compatible boundary. Thus, since $l_k \le p_k \le r_k$ for all $k$ and $\mathbb{E}(\mu_1) + \dots + \mathbb{E}(\mu_N) = C$, we conclude that $\vec{\mu}$ is a step $C$-compatible tuple.
\end{proof}
Let us verify the following simple inequalities.
\begin{proposition}\label{prop:C_le_mean_sum}
If $\vec{\mu}$ is a step $C$-compatible tuple, then
\[
C \le \frac{l_1 + r_1}{2} + \dots + \frac{l_N + r_N}{2}.
\]
Moreover, if the equality is achieved, then $p_k = r_k$ and $\mu_k = \lambda[l_k, r_k]$ for all $k = 1, \dots, N$.
\end{proposition}
\begin{proof}
We have $\mu_k = \alpha_k\lambda[l_k, p_k] + (1 - \alpha_k)\lambda[l_k, r_k]$ for all $k$. Since the function $\mathbb{E}$ is linear and $\mathbb{E}(\lambda[a, b]) = (a + b) / 2$, we have
\begin{equation}\label{eq:E_mu_k_value}
\mathbb{E}(\mu_k) = \alpha_k \frac{l_k + p_k}{2} + (1 - \alpha_k)\frac{l_k + r_k}{2} = \frac{l_k + \alpha_kp_k + (1 - \alpha_k)r_k}{2} \le \frac{l_k + r_k}{2}.
\end{equation}
Summarizing these inequalities for all $k$, we get
\[
C = \mathbb{E}(\mu_1) + \dots + \mathbb{E}(\mu_N) \le \frac{l_1 + r_1}{2} + \dots + \frac{l_N + r_N}{2}.
\]

Assume that the equality holds. Then $\mathbb{E}(\mu_k) = (l_k + r_k) / 2$ for all $k$. Substituting this into equation \cref{eq:E_mu_k_value}, we conclude that $\alpha_kp_k + (1 - \alpha_k)r_k = r_k$, and therefore, since $\alpha_k > 0$, we have $p_k = r_k$ and $\mu_k = \alpha_k\lambda[l_k, r_k] + (1 - \alpha_k)\lambda[l_k, r_k] = \lambda[l_k, r_k]$ for all $k$.
\end{proof}

\begin{proposition}\label{prop:sum_r_ge_C}
If $\vec{\mu}$ is a step $C$-compatible tuple, then the inequality
\[
r_1 + \dots + r_N \ge C + (r_k - l_k)
\]
holds for all $k = 1, \dots, N$. Moreover, if the equality is achieved for some $m$, then $\mu_k = \lambda[l_k, r_k]$ for all $k$ and $r_m - l_m = \sum_{i \ne m}(r_i - l_i)$.
\end{proposition}
\begin{proof}
It follows from \cref{prop:C_le_mean_sum} that
\[
r_1 + \dots + r_N \ge 2C - (l_1 + \dots + l_N).
\]
In addition, since $(\vec{l}, \vec{r})$ is a $C$-compatible boundary, we have
\[
C - (l_1 + \dots + l_N) \ge r_k - l_k
\]
for all $k = 1, \dots, N$. Thus, we conclude that
\[
r_1 + \dots + r_N \ge 2C - (l_1 + \dots + l_N) \ge C + (r_k - l_k).
\]

Assume that the equality holds for some $m$. Then $C = (l_1 + r_1) / 2 + \dots + (l_N + r_N) / 2$ and $r_m - l_m = C - (l_1 + \dots + l_N)$. Substituting the first one into the second one, we get
\[
r_m - l_m = \frac{l_1 + r_1}{2} + \dots + \frac{l_N + r_N}{2} - (l_1 + \dots + l_N) = \frac{r_1 - l_1}{2} + \dots + \frac{r_N - l_N}{2},
\]
or equivalently, $r_m - l_m = \sum_{i \ne m}(r_i - l_i)$. Finally, it follows from \cref{prop:C_le_mean_sum} that $\mu_k = \lambda[l_k, r_k]$ for all $k$.
\end{proof}
\subsection{Full description of the subextreme points of \texorpdfstring{$V^N[C]$}{VN[C]}}
We prove in \cref{prop:2_ex_of_V_C_is_step_C_tuple} that if $\vec{\mu}$ is a subextreme point of $\mathcal{V}^N[C]$, then $\vec{\mu}$ is a step $C$-compatible tuple. This condition is not sufficient; in what follows, we find additional conditions on $\vec{\mu}$ to be a subextreme point of $\mathcal{V}^N[C]$.

In \cref{prop:sum_r_ge_C} we proved that if $\vec{\mu}$ is a step $C$-compatible tuple, then the inequality $r_1 + \dots + r_N \ge C + (r_m - l_m)$ holds for all $m = 1, \dots, N$. In the following proposition we prove that if $\vec{\mu}$ is a subextreme point of $\mathcal{V}^N[C]$ and if $\vec{\mu}$ is not a tuple of Dirac measures, then the inequality is strict.
\begin{proposition}\label{prop:r_sum_eq_not_2ex}
If $\vec{\mu}$ is a step $C$-compatible tuple and if there exists an index $m = 1, \dots, N$ such that $l_m < r_m$ and $r_1 + \dots + r_N = C + (r_m - l_m)$, then $\vec{\mu}$ is not a subextreme point of $\mathcal{V}^N[C]$.
\end{proposition}
\begin{proof}
Since $r_1 + \dots + r_N = C + (r_m - l_m)$, it follows from \cref{prop:sum_r_ge_C} that $\mu_k = \lambda[l_k, r_k]$ for all $k$ and $r_m - l_m = \sum_{i \ne m}(r_i - l_i)$. In addition, since $\mathbb{E}(\mu_k) = (l_k + r_k) / 2$, we have $C = (l_1 + r_1) / 2 + \dots + (l_N + r_N) / 2$.

Denote $t_k = (l_k + r_k) / 2$. Consider the $N$-tuples of probability measures $\vec{\mu}^{\,(1)}$ and $\vec{\mu}^{\,(2)}$ defined as follows:
\begin{align*}
    \mu^{(1)}_k = \begin{cases}
    \lambda[l_m, t_m] &\text{if $k = m$},\\
    \lambda[t_k, r_k] &\text{if $k \ne m$};
    \end{cases}
    && \mu^{(2)}_k = \begin{cases}
    \lambda[t_m, r_m] &\text{if $k = m$},\\
    \lambda[l_k, t_k] &\text{if $k \ne m$}.
    \end{cases}
\end{align*}
Since $\lambda[l_k, r_k] = \frac{1}{2}\lambda[l_k, t_k] + \frac{1}{2}\lambda[t_k, r_k]$ for all $k$, we have $\vec{\mu} = \frac{1}{2}\vec{\mu}^{\,(1)} + \frac{1}{2}\vec{\mu}^{\,(2)}$. In addition, since $l_m < r_m$, we have $\lambda[l_m, t_m] \ne \lambda[t_m, r_m]$, and therefore $\vec{\mu}^{\,(1)} \ne \vec{\mu}^{\,(2)}$.

Let us verify that $\vec{\mu}^{\,(1)}$ satisfies all the assumptions of \cref{cor:lebesgue_tuple_in_V_C_eq}. We have
\begin{align*}
\frac{l_m + t_m}{2} + \sum_{i \ne m}\frac{t_i + r_i}{2} &= \frac{t_1 + \dots + t_N}{2} + \frac{r_1 + \dots + r_N - (r_m - l_m)}{2}\\
&= \frac{1}{2}\left(\frac{l_1 + r_1}{2} + \dots + \frac{l_N + r_N}{2}\right) + \frac{1}{2}C = \frac{1}{2}C + \frac{1}{2}C = C.
\end{align*}
In addition, since $r_m - l_m = \sum_{i \ne m}(r_i - l_i)$, we have
\[
t_m - l_m = \frac{r_m - l_m}{2} = \sum_{i \ne m}\frac{r_i - l_i}{2} = \sum_{i \ne m}(r_i - t_i).
\]
Thus, the $N$-tuple of measures $\vec{\mu}^{\,(1)}$ satisfies all the assumptions of \cref{cor:lebesgue_tuple_in_V_C_eq}, and therefore $\vec{\mu}^{\,(1)} \in \mathcal{V}^N[C]$.

Let us verify in the same manner that $\vec{\mu}^{\,(2)}$ also satisfies all the assumptions of \cref{cor:lebesgue_tuple_in_V_C_eq}. We have
\begin{align*}
l_1 + \dots + l_N + (r_m - l_m) &= l_1 + \dots + l_N + \sum_{i \ne m}(r_i - l_i) \\
&= l_m + \sum_{i \ne m}r_i = r_1 + \dots + r_N - (r_m - l_m) = C,
\end{align*}
and therefore
\[
\frac{t_m + r_m}{2} + \sum_{i \ne m}\frac{l_i + t_i}{2} = \frac{t_1 + \dots + t_N}{2} + \frac{l_1 + \dots + l_N + (r_m - l_m)}{2} = C.
\]
Finally, we have $r_m - t_m = \sum_{i \ne m}(t_i - l_i)$, and therefore $\vec{\mu}^{\,(2)}$ satisfies all the assumptions of \cref{cor:lebesgue_tuple_in_V_C_eq} and $\vec{\mu}^{\,(2)}$ is also contained in $\mathcal{V}^N[C]$. Thus, since $\vec{\mu} = \frac{1}{2}\vec{\mu}^{\,(1)} + \frac{1}{2}\vec{\mu}^{\,(2)}$ and $\vec{\mu}^{\,(1)} \ne \vec{\mu^{\,(2)}}$, we conclude that $\vec{\mu}$ is not a subextreme point of $\mathcal{V}^N[C]$.
\end{proof}
Next we prove that if $\vec{\mu}$ is a subextreme point of $\mathcal{V}^N[C]$, then $\mu_k$ is equal to $\lambda[l_k, r_k]$ for all indices $k = 1, \dots, N$ except for at most one.
\begin{proposition} \label{prop:pk_equal_rk}
If $\vec{\mu}$ is a step $C$-compatible tuple and if $\vec{\mu}$ is a subextreme point of $\mathcal{V}^N[C]$, then $p_k = r_k$ for all $k = 1, \dots, N$ except for at most one.
\end{proposition}
\begin{proof}
Let $\vec{\mu}$ be a step $C$-compatible tuple such that $p_k < r_k$ for at least two indices $k$. Without loss of generality we may assume that $p_{N - 1} < r_{N - 1}$ and $p_N < r_N$.

We have $\mu_{N - 1} = \alpha_{N - 1}\lambda[l_{N - 1}, p_{N - 1}] + (1 - \alpha_{N - 1})\lambda[l_{N - 1}, r_{N - 1}]$ and $\mu_N = \alpha_N\lambda[l_N, p_N] + (1 -\alpha_N)\lambda[l_N, r_N]$. Since $\alpha_{N - 1}, \alpha_N \in (0, 1)$, there exists a positive real number $t$ such that $t(r_{N - 1} - p_{N - 1}) < \min(\alpha_N, 1 - \alpha_N)$ and $t(r_N - p_N) < \min(\alpha_{N - 1}, 1 - \alpha_{N - 1})$. Then denote $\eps_{N - 1} = t(r_{N - 1} - p_{N - 1})$ and $\eps_{N} = t(r_N - p_N)$. Consider the following $N$-tuples of probability measures $\vec{\mu}^{\,(1)}$ and $\vec{\mu}^{\,(2)}$:
\begin{align*}
    \mu_k^{(1)} &= \begin{cases}
    \alpha_k\lambda[l_k, p_k] + (1 - \alpha_k)\lambda[l_k, r_k] &\text{if $k \le N - 2$},\\
    (\alpha_{N - 1} - \eps_N)\lambda[l_{N - 1}, p_{N - 1}] + (1 - \alpha_{N - 1} + \eps_N)\lambda[l_{N - 1}, r_{N - 1}] &\text{if $k = N - 1$},\\
    (\alpha_N + \eps_{N - 1})\lambda[l_N, p_N] + (1 - \alpha_N - \eps_{N - 1})\lambda[l_N, r_N] &\text{if $k = N$},
    \end{cases}
    \intertext{and similarly}
    \mu_k^{(2)} &= \begin{cases}
    \alpha_k\lambda[l_k, p_k] + (1 - \alpha_k)\lambda[l_k, r_k] &\text{if $k \le N - 2$},\\
    (\alpha_{N - 1} + \eps_N)\lambda[l_{N - 1}, p_{N - 1}] + (1 - \alpha_{N - 1} - \eps_N)\lambda[l_{N - 1}, r_{N - 1}] &\text{if $k = N - 1$},\\
    (\alpha_N - \eps_{N - 1})\lambda[l_N, p_N] + (1 - \alpha_N + \eps_{N - 1})\lambda[l_N, r_N] &\text{if $k = N$},
    \end{cases}
\end{align*}

One can easily verify that $\vec{\mu} = \frac{1}{2}\vec{\mu}^{\,(1)} + \frac{1}{2}\vec{\mu}^{\,(2)}$. Since $\eps_{N - 1} \ne 0$ and $\lambda[l_N, p_N] \ne \lambda[l_N, r_N]$, we conclude that $\mu_N \ne \mu^{(1)}_N$, and therefore $\vec{\mu} \ne \vec{\mu}^{\,(1)}$. Hence, this convex combination is nontrivial.

Let us verify that both $N$-tuples $\vec{\mu}^{\,(1)}$ and $\vec{\mu}^{\,(2)}$ are contained in $\mathcal{V}^N[C]$. By construction, $\mu_k^{(1)} \in \mathcal{D}[l_k, r_k]$ for all $k = 1, \dots, N$. In addition, \begin{align*}
    \mathbb{E}(\mu_{N - 1}^{(1)}) = \mathbb{E}(\mu_{N - 1}) + \frac{\eps_N}{2}(r_{N - 1} - l_{N - 1}), && \mathbb{E}(\mu_N^{(1)}) = \mathbb{E}(\mu_N) - \frac{\eps_{N - 1}}{2}(r_N - p_N),
\end{align*}
and therefore, since $\eps_{N}(r_{N - 1} - p_{N - 1}) = \eps_{N - 1}(r_N - p_N) = t(r_{N - 1} - p_{N - 1})(r_N - p_N)$ and $\mathbb{E}(\mu^{(1)}_k) = \mathbb{E}(\mu_k)$ for all $k \le N - 2$, we conclude that
\[
\mathbb{E}(\mu_1^{(1)}) + \dots + \mathbb{E}(\mu_N^{(1)}) = \mathbb{E}(\mu_1) + \dots + \mathbb{E}(\mu_N) = C.
\]

Thus, $\vec{\mu}^{\,(1)} \in \mathcal{V}^N[C]$. One can prove in the same manner that $\vec{\mu}^{\,(2)}$ is also contained in $\mathcal{V}^N[C]$, and therefore $\vec{\mu}$ is not a subextreme point of $\mathcal{V}^N[C]$.
\end{proof}

The following condition on $\vec{\mu}$ to be a subextreme point of $\mathcal{V}^N[C]$ is technical. 
\begin{proposition}\label{prop:cut_from_the_right}
Let $\vec{\mu}$ be a step $C$-compatible tuple. Suppose that there exists an index $m = 1, \dots, N$ such that $r_k - l_k < C - (l_1 + \dots + l_N)$ for all $k \ne m$ and that
\begin{equation}\label{eq:C_strange_ineq}
C > \frac{l_1 + r_1}{2} + \dots + \frac{l_N + l_N}{2} - \frac{r_m - l_m}{2}.
\end{equation}
Then $\vec{\mu}$ is not a subextreme point of $\mathcal{V}^N[C]$.
\end{proposition}

To prove this statement, we first verify the following lemmas.

\begin{lemma}\label{lem:xi_tuple_exists}
Let $(\vec{l}, \vec{r})$ be a $C$-compatible boundary. Suppose that the inequality $r_k - l_k < C - (l_1 + \dots +l_N)$ holds for all $k \le N - 1$ and that
\[
r_1 + \dots + r_{N - 1} + l_N > C > \frac{l_1 + r_1}{2} + \dots + \frac{l_{N - 1} + r_{N - 1}}{2} + l_N.
\]
Then there exists a tuple of real numbers $\{\xi_k\}_{k = 1}^{N - 1}$ that satisfies the following conditions:
\begin{conditions}
    \item\label{cond:xi_bounds} the inequality $0 \le \xi_k \le r_k - l_k$ holds for all $k \le N - 1$;
    \item\label{cond:xi_strictly_positive} if $l_k < r_k$, then $\xi_k > 0$;
    \item\label{cond:xi_form_broken_line} the inequality $2\xi_k \le \xi_1 + \dots + \xi_{N - 1}$ holds for all $k \le N - 1$;
    \item\label{cond:xi_strange_ineq} $r_1 + \dots + r_{N - 1} + l_N < C + (\xi_1 + \dots + \xi_{N - 1}) / 2$.
\end{conditions}
\end{lemma}
\begin{proof}
We consider the following two cases.

\textit{Case 1.} Suppose that the inequality $2(r_k - l_k) \le (r_1 - l_1) + \dots + (r_{N - 1} - l_{N - 1})$ holds for all $k = 1, \dots, N- 1$. Then we put $\xi_k = r_k - l_k$. Checking \cref{cond:xi_bounds,cond:xi_strictly_positive,cond:xi_form_broken_line} is trivial. In addition,
\begin{align*}
C + \frac{\xi_1 + \dots + \xi_{N - 1}}{2} &> \left(\frac{l_1 + r_1}{2} + \dots + \frac{l_{N - 1} + r_{N - 1}}{2} + l_N\right) + \frac{\xi_1 + \dots + \xi_{N - 1}}{2} \\
&= r_1 + \dots + r_{N - 1} + l_N,
\end{align*}
and this implies \cref{cond:xi_strange_ineq}.

\textit{Case 2.} Suppose that there exists an index $m = 1, \dots, N - 1$ such that $2(r_m - l_m) > (r_1 - l_1) + \dots + (r_{N - 1} - l_{N - 1})$. Without loss of generality we may assume that $m = N - 1$. Then put $\xi_k = r_k - l_k$ for all $k = 1, \dots, N - 2$ and put $\xi_{N - 1} = (r_1 - l_1) + \dots + (r_{N - 2} - l_{N - 2})$.

First, by construction the inequality $0 \le \xi_k \le r_k - l_k$ is trivial for all $k \le N - 2$. The inequality $2(r_{N - 1} - l_{N - 1}) > (r_1 - l_1) + \dots + (r_{N - 1} - l_{N - 1})$ is equivalent to \[r_{N - 1} - l_{N - 1} > (r_1 - l_1) + \dots + (r_{N - 2} + r_{N - 2}) = \xi_{N - 1}.\]
Thus, $0 \le \xi_{N - 1} < r_{N - 1} - l_{N - 1}$, and this implies \cref{cond:xi_bounds}.

\Cref{cond:xi_strictly_positive} is trivial for all $k \le N - 2$. Since $(\vec{l}, \vec{r})$ is a $C$-compatible boundary, we have $r_{N - 1} - l_{N - 1} \le C - (l_1 + \dots + l_N)$, or equivalently $l_1 + \dots + l_{N - 2} + r_{N - 1} + l_N \le C$. In addition, since $C < r_1 + \dots + r_{N - 1} + l_N$, we conclude that
\[
r_1 + \dots + r_{N - 1} + l_{N} > l_1 + \dots + l_{N - 2} + r_{N - 1} + l_N,
\]
which is equivalent to $r_1 + \dots + r_{N - 2} > l_1 + \dots + l_{N - 2}$. Thus, \[
\xi_{N - 1} = (r_1 + \dots + r_{N - 2}) - (l_1 + \dots + l_{N - 2}) > 0,\]
and this implies \cref{cond:xi_strictly_positive}.

\Cref{cond:xi_form_broken_line} is equivalent to the inequality $2\max\{\xi_k \colon 1 \le k \le N - 1\} \le \xi_1 + \dots + \xi_{N - 1}$. Since $\xi_{N - 1} = \xi_1 + \dots + \xi_{N - 2}$ and all $\xi_k$ are nonnegative, we conclude that $\max\{\xi_k \colon 1 \le k \le N - 1\} = \xi_{N - 1}$. The equation $2\xi_{N - 1} = \xi_1 + \dots + \xi_{N - 1}$ follows trivially by construction.

We have $(\xi_1 + \dots + \xi_{N - 1}) / 2 = (r_1 - l_1) + \dots + (r_{N - 2} - l_{N - 2})$. Hence, the inequality \[
r_1 + \dots + r_{N - 1} + l_N < C + (\xi_1 + \dots + \xi_{N - 1}) / 2\]
is equivalent to $r_{N - 1} - l_{N - 1} < C - (l_1 + \dots + l_N)$, which follows from the lemma assumption. Thus, \cref{cond:xi_strange_ineq} is achieved.
\end{proof}

\begin{lemma} \label{lem:zeta_tuple_exists}
Let $(\vec{l}, \vec{r})$ be a $C$-compatible boundary. Suppose that the inequality $r_k - l_k < C - (l_1 + \dots +l_N)$ holds for all $k \le N - 1$ and that
\[
r_1 + \dots + r_{N - 1} + l_N > C > \frac{l_1 + r_1}{2} + \dots + \frac{l_{N - 1} + r_{N - 1}}{2} + l_N.
\]
Then there exists a real number $\eps > 0$ and a tuple of functions $\{\zeta_k(t)\}_{k = 1}^{N - 1}$ defined on the segment $[0, \eps]$ such that the following conditions hold:
\begin{conditions}
    \item \label{cond:zeta_nonincreasing}the function $\zeta_k(t)$ is nonincreasing for all $k \le N - 1$;
    \item \label{cond:zeta_inside_segment}the inequality $l_k \le \zeta_k(t) \le r_k$ holds for all $t \in [0, \eps]$ and for all $k \le N - 1$;
    \item \label{cond:zeta_less_rk} if $l_k < r_k$, then $\zeta_k(t) < r_k$ for all $t \in [0, \eps]$;
    \item \label{cond:zeta_measure_in_V_C}the $N$-tuple of probability measures \[(\lambda[\zeta_1(t), r_1], \dots, \lambda[\zeta_{N - 1}(t), r_{N - 1}], \lambda[l_N, l_N + t])\]
    is contained in $\mathcal{V}^N[C]$ for all $t \in [0, \eps]$.
\end{conditions}
\end{lemma}
\begin{proof}
We are under the assumptions of \cref{lem:xi_tuple_exists}, and therefore there exists a tuple of real numbers $\{\xi_k^{(0)}\}_{k = 1}^{N - 1}$ that satisfies \cref{cond:xi_bounds,cond:xi_strictly_positive,cond:xi_form_broken_line,cond:xi_strange_ineq}.

By \cref{cond:xi_strange_ineq}, 
\begin{equation}\label{eq:sum_xi_greater_than_0}
0 < 2(r_1 + \dots + r_{N - 1} - C) < \xi_1^{(0)} + \dots + \xi_{N - 1}^{(0)},
\end{equation}
and therefore we conclude that $\xi^{(0)}_1 + \dots + \xi^{(0)}_{N - 1} > 0$. For each $k \le N - 1$ consider the function $\xi_k(t)$ defined on the set of nonnegative real numbers as follows:
\begin{equation}\label{eq:xi_t_definition}
\xi_k(t) = \xi_k^{(0)} \cdot \frac{2(r_1 + \dots + r_{N - 1} + l_N - C) + t}{\xi_1^{(0)} + \dots + \xi_{N - 1}^{(0)}}
\end{equation}
Put $\zeta_k(t) = r_k - \xi_k(t)$. We claim that the restrictions of these functions to the segment $[0, \eps]$ for some $\eps > 0$ satisfy all the conditions.

Since $\xi_k^{(0)} \ge 0$ and the function $2(r_1 + \dots + r_{N - 1} + l_N - C) + t$ is increasing, the function $\xi_k(t)$ is nondecreasing. Hence, the function $\zeta_k(t) = r_k - \xi_k(t)$ is nonincreasing, and this implies \cref{cond:zeta_nonincreasing}.

Denote $\eps^{(1)} = \xi_1^{(0)} + \dots + \xi_{N - 1}^{(0)} - 2(r_1 + \dots + r_{N - 1} - C)$. By inequality \cref{eq:sum_xi_greater_than_0} we have $\eps^{(1)} > 0$. In addition, it follows from equation \cref{eq:xi_t_definition} that if $0 \le t \le \eps^{(1)}$, then $0 \le \xi_k(t) \le \xi_k^{(0)}$. Since $\xi_k^{(0)} \le r_k - l_k$ by \cref{cond:xi_bounds}, we conclude that $0 \le \xi_k(t) \le r_k - l_k$, and therefore $l_k \le \zeta_k(t) \le r_k$ for all $t \in [0, \eps^{(1)}]$, and this implies \cref{cond:zeta_inside_segment}.

In addition, if $l_k < r_k$, then by \cref{cond:xi_strictly_positive} we have $\xi_k^{(0)} > 0$, and therefore $\xi_k(t) > 0$ for all $t \ge 0$. Thus, if $l_k < r_k$, then $\zeta_k(t) = r_k - \xi_k(t) < r_k$, and therefore \cref{cond:zeta_less_rk} is also achieved.

We claim that the $N$-tuple of measures $\vec{\mu}(t)$ defined in \cref{cond:zeta_measure_in_V_C} satisfies the conditions of \cref{prop:lebesgue_tuple_in_V_C_ineq}. It follows from equation \cref{eq:xi_t_definition} that
\begin{equation}\label{eq:xi_t_sum}
\xi_1(t) + \dots + \xi_{N - 1}(t) - t = 2(r_1 + \dots + r_{N - 1} + l_N - C),
\end{equation}
and therefore
\begin{align*}
    \frac{l_N + (l_N + t)}{2} + \sum_{k = 1}^{N - 1}\frac{\zeta_k(t) + r_k}{2} = (r_1 + \dots + r_{N - 1} + l_N) - \frac{\xi_1(t) + \dots + \xi_{N - 1}(t) - t}{2} = C.
\end{align*}

Denote $\xi_N(t) = t$. The support length of the $k$th item of $\vec{\mu}(t)$ is equal to $\xi_k(t)$ for all $k = 1, \dots, N$. By \cref{prop:lebesgue_tuple_in_V_C_ineq}, to prove that $\vec{\mu}(t) \in \mathcal{V}^N[C]$ it suffices to check that the inequality $\xi_k(t) \le \sum_{i \ne k}\xi_i(t)$ holds for all $k \le N$, or equivalently, \[2\max\{\xi_k(t) \colon 1 \le k \le N\} \le \xi_1(t) + \dots + \xi_N(t).\]

Let $\eps^{(2)} = \max\{\xi_k(0) \colon 1 \le k \le N - 1\}$. It follows from \cref{eq:xi_t_sum} that
\[
\xi_1(0) + \dots + \xi_{N - 1}(0) = 2(r_1 + \dots + r_{N - 1} + l_N - C) > 0,
\]
and therefore $\eps^{(2)} > 0$. Since the functions $\xi_k(t)$ are nondecreasing, we conclude that if $0 \le t \le \eps^{(2)}$, then $\xi_N(t) = t \le \max\{\xi_k(t) \colon 1 \le k \le N - 1\}$. Hence, if $0 \le t \le \eps^{(2)}$, then
\[
\max\{\xi_k(t) \colon 1 \le k \le N\} = \max\{\xi_k(t) \colon 1 \le k \le N - 1\}.
\]

The functions $\xi_1(t), \dots, \xi_{N - 1}(t)$ are proportional to $\xi_1^{(0)}, \dots, \xi_{N - 1}^{(0)}$ with the same positive coefficient, and therefore it follows from \cref{cond:xi_form_broken_line} that the inequality
\[
2\max\{\xi_k(t) \colon 1 \le k \le N - 1\} \le \xi_1(t) + \dots + \xi_{N - 1}(t) \le \xi_1(t) + \dots + \xi_N(t)
\]
holds for all $t \ge 0$. Thus, if $0 \le t \le \eps^{(2)}$, then $2\max\{\xi_k(t) \colon 1 \le k \le N\} \le \xi_1(t) + \dots + \xi_N(t)$, and therefore if $0 \le t \le \eps^{(2)}$, then by \cref{prop:lebesgue_tuple_in_V_C_ineq} the $N$-tuple of probability measures $\vec{\mu}(t)$ is contained in $\mathcal{V}^N[C]$. Hence, if we put $\eps = \min(\eps^{(1)}, \eps^{(2)})$, then the restrictions of the functions $\zeta_k(t)$ to the segment $[0, \eps]$ satisfy all \cref{cond:zeta_nonincreasing,cond:zeta_inside_segment,cond:zeta_less_rk,cond:zeta_measure_in_V_C}.
\end{proof}

\begin{proof}[Proof of \cref{prop:cut_from_the_right}]

By \cref{prop:C_le_mean_sum},
\[
\frac{l_1 + r_1}{2} + \dots + \frac{l_N + r_N}{2} \ge C,
\]
and combining this with inequality \cref{eq:C_strange_ineq} we conclude that $r_m > l_m$. By \cref{prop:sum_r_ge_C}, $r_1 + \dots + r_N \ge C + (r_m - l_m)$. If the equality is achieved, then $\vec{\mu} \notin \mathrm{se}\,\mathcal{V}^N[C]$ by \cref{prop:r_sum_eq_not_2ex}. Thus, the following inequality holds:
\begin{equation}\label{eq:sum_rm_ge_C}
r_1 + \dots + r_N - (r_m - l_m) > C > \frac{l_1 + r_1}{2} + \dots + \frac{l_N + r_N}{2} - \frac{r_m - l_m}{2}.
\end{equation}

Without loss of generality we may assume that $m = N$. Then the inequality $r_k- l_k < C - (l_1 + \dots + l_N)$ holds for all $k \le N - 1$, $l_N < r_N$ and inequality \cref{eq:sum_rm_ge_C} transforms into 
\[
r_1 + \dots + r_{N - 1} + l_N > C > \frac{l_1 + r_1}{2} + \dots + \frac{l_{N - 1} + r_{N - 1}}{2} + l_N.
\]
Thus, all the assumptions of \cref{lem:zeta_tuple_exists} hold, and therefore there exists a real number $\eps^{(1)} > 0$ and a tuple of functions $\{\zeta_k\}_{k = 1}^{N - 1}$ defined on the segment $[0, \eps^{(1)}]$ satisfying \cref{cond:zeta_nonincreasing,cond:zeta_inside_segment,cond:zeta_less_rk,cond:zeta_measure_in_V_C}.

We claim that there exists a real number $A > 0$ such that if $0 \le \alpha \le A$, then the inequality 
\begin{equation}\label{eq:lambda_le_mu_k}
\alpha\lambda[\zeta_k(t), r_k] \le \mu_k
\end{equation}
holds for all $t \in [0, \eps^{(1)}]$ and for all $k \le N - 1$. Since $\mu_k = \alpha_k\lambda[l_k, p_k] + (1 - \alpha_k)\lambda[l_k, r_k]$, the inequality \cref{eq:lambda_le_mu_k} follows from 
\begin{equation}\label{eq:lambda_le_part_mu_k}
\alpha\lambda[\zeta_k(t), r_k] \le (1 - \alpha_k)\lambda[l_k, r_k].
\end{equation}

By \cref{cond:zeta_inside_segment} we have $l_k \le \zeta_k(t) \le r_k$. If $l_k = r_k$, then $\lambda[l_k, r_k] = \lambda[\xi_k(t), r_k] = \delta(l_k)$, and inequality \cref{eq:lambda_le_part_mu_k} is equivalent to $\alpha \le (1 - \alpha_k)$. 

Suppose otherwise that $l_k < r_k$. Then by \cref{cond:zeta_less_rk} we have $l_k \le \xi_k(t) < r_k$. Hence, inequality \cref{eq:lambda_le_part_mu_k} holds if and only if the density function of the left hand-side is not greater than the density function of the right hand-side, or equivalently,
\[
\frac{\alpha}{r_k - \zeta_k(t)} \le \frac{1 - \alpha_k}{r_k - l_k}.
\]

Since the function $\zeta_k(t)$ is nonincreasing by \cref{cond:zeta_nonincreasing}, we have $r_k - \zeta_k(t) \ge r_k - \zeta_k(0) > 0$, and therefore the previous inequality follows from
\[
\alpha \le (1 - \alpha_k)\frac{r_k - \zeta_k(0)}{r_k - l_k}.
\]

Thus, put
\[
A = \min_{1 \le k \le {N - 1}}(1 - \alpha_k)\frac{r_k - \zeta_k(0)}{r_k - l_k},
\]
where $(r_k - \zeta_k(0)) / (r_k - l_k) = 1$  if $l_k = r_k$. Since $\alpha_k < 1$, we conclude that $A > 0$. In addition, if $0 \le \alpha \le A$, then the inequality $\alpha\lambda[\zeta_k(t), r_k] \le \mu_k$ holds for all $t \in [0, \eps^{(1)}]$ and for all $k \le N - 1$.

Denote \[
\vec{\mu}^{\,(1)}(t) =  \left(\lambda[\zeta_1(t), r_1], \dots, \lambda[\zeta_{N - 1}(t), r_{N - 1}], \lambda[l_N, l_N + t]\right).\] 
It follows from \cref{cond:zeta_measure_in_V_C} that $\vec{\mu}^{\,(1)}(t) \in \mathcal{V}^N[C]$ for all $t \in [0, \eps^{(1)}]$. We may assume that $A < 1$ and that $0 < \alpha \le A$. Denote \[
\vec{\mu}^{\,(2)}(t, \alpha) = (\vec{\mu} - \alpha\vec{\mu}^{\,(1)}(t)) / (1 - \alpha).\]
We have $\vec{\mu} = \alpha\vec{\mu}^{\,(1)}(t) + (1 - \alpha)\vec{\mu}^{\,(2)}(t, \alpha)$. Since $\vec{\mu} \in \mathcal{V}^N[C]$ and $\vec{\mu}^{\,(1)}(t) \in \mathcal{V}^N[C]$, we have \[\mathbb{E}(\mu_1) + \dots + \mathbb{E}(\mu_N) = \mathbb{E}(\mu_1^{(1)}(t)) + \dots + \mathbb{E}(\mu_N^{(1)}(t)) = C,\] 
and therefore it follows from the linearity of $\mathbb{E}$ that
\begin{equation}\label{eq:mu_2_expectation_sum}
\mathbb{E}(\mu_1^{(2)}(t, \alpha)) + \dots + \mathbb{E}(\mu_N^{(2)}(t, \alpha)) = C.
\end{equation}
In addition, since $\alpha\mu_k^{(1)}(t) = \alpha\lambda[\zeta_k(t), r_k] \le \mu_k$ and $\mu_k \in \mathcal{D}[l_k, r_k]$ for all $k \le N - 1$, we conclude that
\[
\mu_k^{(2)}(t, \alpha) = \frac{\mu_k - \alpha\lambda[\zeta_k(t), r_k]}{1 - \alpha} \in \mathcal{D}[l_k, r_k]
\]
for all $k \le N - 1$. See \cref{fig:cut_right_part_visualization} for visualization. Next, we consider two following cases.
\begin{figure}[!htb]
    \begin{minipage}{0.49\textwidth}
    \centering

    \begin{tikzpicture}[xscale=0.8,yscale=0.5]
    \coordinate[label=below:$l_k$] (LD) at (0, 0);
    \coordinate (LU) at (0, 5);
    \coordinate (PU) at (4, 5);
    \coordinate (PM) at (4, 3);
    \coordinate[label=below:$p_k$] (PD) at (4, 0);
    \coordinate (RM) at (8, 3);
    \coordinate[label=below:$r_k$] (RD) at (8, 0);

    \draw[black,fill=yellow, thin] (LD) -- (LU) -- (PU) -- (PM) -- (RM) -- (RD) -- cycle;
    \coordinate[label=below:$\zeta_k(t)$] (ZD) at (2.5, 0);
    \coordinate (ZU) at (2.5, 5);
    \coordinate (ZM) at (2.5, 1);
    \coordinate (RZ) at (8, 1);
    \draw[draw=none, fill=red, very thin] (ZD) -- (ZM) -- (RZ) -- (RD) -- cycle;

    \draw[dashed] (PM) -- (PD);
    %\draw[dashed] (ZU) -- (ZD);
    \end{tikzpicture}

    \end{minipage}\hfill
        \begin{minipage}{0.49\textwidth}
    \centering
    \begin{tikzpicture}[xscale=0.8,yscale=0.5]
    \coordinate[label=below:$l_N$] (LD) at (0, 0);
    \coordinate (LU) at (0, 5);
    \coordinate (PU) at (4, 5);
    \coordinate (PM) at (4, 3);
    \coordinate[label=below:$p_N$] (PD) at (4, 0);
    \coordinate (RM) at (8, 3);
    \coordinate[label=below:$r_N$] (RD) at (8, 0);

    \draw[black,fill=yellow, thin] (LD) -- (LU) -- (PU) -- (PM) -- (RM) -- (RD) -- cycle;
    \coordinate[label=below right:$l_N + t$] (ZD) at (0.7, 0);
    \coordinate (ZU) at (0.7, 5);
    \draw[draw=none, fill=red, very thin] (LD) -- (LU) -- (ZU) -- (ZD) -- cycle;

    \draw[dashed] (PM) -- (PD);
    \end{tikzpicture}
    \end{minipage}
        \caption{Decomposition of the measure $\mu_k = \alpha_k\lambda[l_k, p_k] + (1 - \alpha_k)\lambda[l_k, r_k]$ into the convex combination $\alpha \mu_k^{(1)}(t) + (1 - \alpha) \mu_k^{(2)}(t, \alpha)$. The measure $\alpha\mu_k^{(1)}(t)$ is colored red, and the measure $(1 - \alpha)\mu_k^{(2)}(t, \alpha)$ is yellow. 
        The case $k \le N - 1$ is shown on the left part of the figure: here, $\mu_k^{(1)}(t) = \lambda[\zeta_k(t), r_k]$. The case $k = N$ is shown on the right part of the figure: in this case, $\mu_N^{(1)} = \lambda[l_N, l_N + t]$.
        }
    \label{fig:cut_right_part_visualization}
\end{figure}
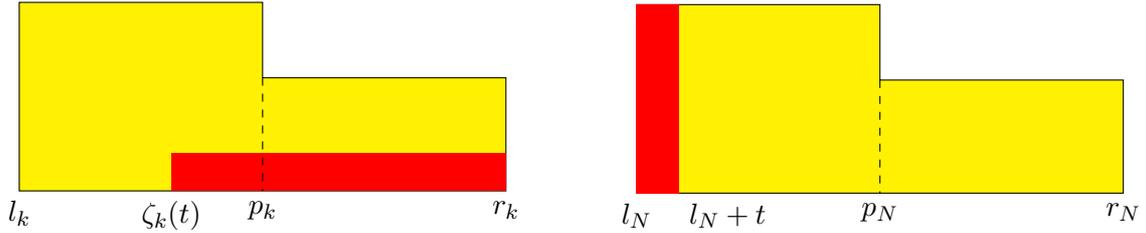

\textit{Case 1.} Suppose that $p_N = l_N$, and therefore $\mu_N = \alpha_N\delta(l_N) + (1 - \alpha_N)\lambda[l_N, r_N]$, where $0 < \alpha_N < 1$. Put $t_0 = 0$ and $\alpha_0 = \min(\alpha_N, A)$. Then $\mu_N^{(1)}(t_0) = \lambda[l, l] = \delta(l)$, and therefore
\[
\mu_N^{(2)}(t_0, \alpha_0) = \frac{(\alpha_N - \alpha_0)\delta(l) + (1 - \alpha_N)\lambda[l_N, r_N]}{1 - \alpha_0} \in \mathcal{D}[l_N, r_N].
\]

Since in addition $\mu_k^{(2)}(t, \alpha) \in \mathcal{D}[l_k, r_k]$ for all $k \le N - 1$, we conclude that $\vec{\mu}^{\,(2)}(t, \alpha) \in \mathcal{V}^N[C]$. Finally, since $\supp(\mu_N^{(1)}(t, \alpha)) = \{l_N\}$, $\supp(\mu_N) = [l_N, r_N]$ and $l_N < r_N$, we conclude that $\vec{\mu}^{\,(1)}(t, \alpha) \ne \vec{\mu}$. Thus, $\vec{\mu}$ is a nontrivial convex combination of $\vec{\mu}^{\,(1)}(t), \vec{\mu}^{\,(2)}(t, \alpha) \in \mathcal{V}^N[C]$, and therefore $\vec{\mu}$ is not a subextreme point of $\mathcal{V}^N[C]$ in this case.

\textit{Case 2.} Suppose that $l_N < p_N$. First, since $r_N - l_N > 0$ and $r_k - l_k < C - (l_1 + \dots + l_N)$ for all $k \le N - 1$, there exists a real number $\eps^{(2)} > 0$ such that the inequalities $r_N - l_N - t \ge 0$ and $r_k - l_k \le C - (l_1 + \dots + l_N + t)$ hold for all $k \le N -1$ and for all $t \le \eps^{(2)}$. Equivalently, if we denote $\vec{l}^{\,(2)}(t) = (l_1, \dots, l_{N - 1}, l_N + t)$, then the pair of $N$-tuples of point $(\vec{l}^{\,(2)}(t), \vec{r}\,)$ is a $C$-compatible boundary for all $t \in [0, \eps^{(2)}]$.

We have $\mu_N = \alpha_N\lambda[l_N, p_N] + (1 - \alpha_N)\lambda[l_N, r_N)$. One can easily verify that for all $t \in (0, p_N - l_N)$ there exists a real number $\alpha(t) \in (0, 1)$ such that $\mu_N$ can be represented in the form 
\begin{equation}\label{eq:alpha_t_definition}
\mu_N = \alpha(t) \lambda[l_N, l_N + t] + (1 - \alpha(t))\mu',
\end{equation}
where $\mu' \in \mathcal{D}[l_N + t, r_N]$; moreover, if $t \to 0$, then $\alpha(t) \to 0$. Hence, there exists a point $t_0$ such that $0 < t_0 < p_N - l_N$, $t_0 < \eps^{(1)}$, $t_0 < \eps^{(2)}$ and $\alpha_0 = \alpha(t_0) < A$.

Since $0 < t_0 < \eps^{(1)}$ and $\alpha_0 < A$, we conclude that $\mu_k^{(2)}(t_0, \alpha_0) \in \mathcal{D}[l_k, r_k]$ for all $k \le N - 1$. In addition, since $0 < t_0 < p_N - l_N$ and $\alpha_0 = \alpha(t_0)$, it follows from equation \cref{eq:alpha_t_definition} that $\mu_N^{(2)}(t_0, \alpha_0) \in \mathcal{D}[l_N + t_0, r_N]$. Finally, since $t_0 < \eps^{(2)}$, we conclude that $(\vec{l}^{\,(2)}(t_0), \vec{r}\,)$ is a $C$-compatible boundary, and combining this with equation \cref{eq:mu_2_expectation_sum} we get $\vec{\mu}^{\,(2)}(t_0, \alpha_0) \in \mathcal{V}^N[C]$.

Since $\supp(\mu^{(1)}_N(t_0)) = [l_N, l_N + t_0]$ and $\supp(\mu_N) = [l_N, r_N] \ne [l_N, l_N + t_0]$, we conclude that $\vec{\mu}^{\,(1)}(t_0) \ne \vec{\mu}$. Hence, $\vec{\mu}$ can be represented in the form
\[
\vec{\mu} = \alpha_0\vec{\mu}^{\,(1)}(t_0) + (1 - \alpha_0)\vec{\mu}^{\,(2)}(t_0, \alpha_0),
\]
where $0 < \alpha_0 < 1$, both $N$-tuples of probability measures $\vec{\mu}^{\,(1)}(t_0)$ and $\vec{\mu}^{\,(2)}(t_0, \alpha_0)$ are contained in $\mathcal{V}^N[C]$, and $\vec{\mu}^{\,(1)}(t_0) \ne \vec{\mu}$. Thus, $\vec{\mu}$ is not a subextreme point of $\mathcal{V}^N[C]$.
\end{proof}

\begin{theorem}\label{thm:2_ex_of_V_C}
An $N$-tuple of probability measures $\vec{\mu}$ is a subextreme point of $\mathcal{V}^N[C]$ if and only if $\mu_k = \delta(l_k)$ for all $k = 1, \dots, N$, where $l_1, \dots, l_N$ are some points on the real line such that $l_1 + \dots + l_N = C$.
\end{theorem}
\begin{proof}
Let $\vec{\mu}$ be a subextreme point of $\mathcal{V}^N[C]$. By \cref{prop:2_ex_of_V_C_is_step_C_tuple} we have $\vec{\mu}$ is a step $C$-compatible tuple, and therefore $\mu_k = \alpha_k\lambda[l_k, p_k] + (1 - \alpha_k)\lambda[l_k, r_k]$ for all $k = 1, \dots, N$, where $(\vec{l}, \vec{r})$ is a $C$-compatible boundary, $l_k \le p_k \le r_k$ and $\alpha_k \in (0, 1)$ for all $k$. We claim that $l_k = r_k$ for all $k$.

Assume the converse. Without loss of generality we may assume that \[
r_1 - l_1 \le r_2 - l_2 \le \dots \le r_N - r_N.\]
Since there exists an index $k$ such that $l_k < r_k$, we conclude that $l_N < r_N$.

It follows from \cref{prop:pk_equal_rk} that there exists an index $m = 1, \dots, N$ such that if $k \ne m$, then $p_k = r_k$. Hence, if $k \ne m$, then $\mu_k = \lambda[l_k, r_k]$, and therefore $\mathbb{E}(\mu_k) = (l_k + r_k) / 2$. In addition, $
\mathbb{E}(\mu_m) = (l_m + \alpha_mp_m + (1 - \alpha_m)r_m) / 2$, and therefore
\begin{equation}\label{eq:C_repr_with_pm_rm}
C = \mathbb{E}(\mu_1) + \dots + \mathbb{E}(\mu_N) = \frac{l_1 + r_1}{2} + \dots + \frac{l_N + r_N}{2} - \frac{\alpha_m(r_m - p_m)}{2}.
\end{equation}

Since $r_N - l_N > 0$ and $\alpha_m < 1$, we have 
\[
\frac{r_N - l_N}{2} > \frac{\alpha_m(r_N - l_N)}{2} \ge \frac{\alpha_m(r_m - l_m)}{2} \ge \frac{\alpha_m(r_m - p_m)}{2},
\]
and therefore it follows from equation \cref{eq:C_repr_with_pm_rm} that
\[
C > \frac{l_1 + r_1}{2} + \dots + \frac{l_N + r_N}{2} - \frac{r_N - l_N}{2}.
\]

If $r_{N - 1} - l_{N - 1} < C - (l_1 + \dots + l_N)$, then $r_k - l_k < C - (c_1 + \dots + l_N)$ for all $k \le N - 1$ and then by \cref{prop:cut_from_the_right} the $N$-tuple of probability measures $\vec{\mu}$ is not a subextreme point of $\mathcal{V}^N[C]$. Hence,
\begin{equation}\label{eq:rn_ln_eq_C_minus_sum_l}
r_{N - 1} - l_{N - 1} = r_N - l_N = C - (l_1 + \dots + l_N).
\end{equation}

Since $C = \mathbb{E}(\mu_1) + \dots + \mathbb{E}(\mu_N)$, we have
\begin{equation}\label{eq:rn_ln_eq_E_minus_sum_l}
r_{N - 1} - l_{N - 1} = r_N - l_N = (\mathbb{E}(\mu_1) - l_1) + \dots + (\mathbb{E}(\mu_N) - l_N),
\end{equation}
where all the summands in the right hand-side of this equation are nonnegative. Suppose that $p_{N - 1} = r_{N - 1}$ and $p_N = r_N$. Then $\mu_{N - 1} = \lambda[l_{N - 1}, r_{N - 1}]$, and therefore $\mathbb{E}(\mu_{N - 1}) - l_{N - 1} = (r_{N - 1} - l_{N - 1}) / 2$, and similarly $\mathbb{E}(\mu_N) - l_N = (r_N - l_N) / 2 = (r_{N - 1} - l_{N - 1}) / 2$. Substituting this into equation \cref{eq:rn_ln_eq_E_minus_sum_l} we get
\[
r_{N - 1} - l_{N - 1} = (\mathbb{E}(\mu_1) - l_1) + \dots + (\mathbb{E}(\mu_{N - 2}) - l_{N - 2}) + (r_{N - 1} - l_{N - 1}),
\]
and therefore we conclude that $l_k = \mathbb{E}(\mu_k)$ for all $k \le N - 2$.

Since $\mathbb{E}(\mu_k) = (l_k + \alpha_kp_k + (1 - \alpha_k)r_k) / 2$ and $0 < \alpha_k < 1$, we have $l_k = p_k = r_k$ for all $k \le N - 2$, or equivalently $\mu_k = \delta(l_k)$ for all $k \le N - 2$. Hence,
\[
r_{N - 1} - l_{N - 1} = C - (l_1 + \dots + l_N) = C - (r_1 + \dots + r_{N - 2}) - l_{N - 1} - l_{N},
\]
or equivalently $r_1 + \dots + r_N = C + (r_N - l_N)$. Then, since in addition $l_N < r_N$, it follows from \cref{prop:r_sum_eq_not_2ex} that $\vec{\mu}$ is not a subextreme point of $\mathcal{V}^N[C]$ in the case of $p_{N - 1} = r_{N - 1}$ and $p_N = r_N$.

Thus, either $p_{N - 1} < r_{N - 1}$ or $p_N < r_N$. Without loss of generality we may assume that $p_N < r_N$. Then it follows from \cref{prop:pk_equal_rk} that $\mu_k = \lambda[l_k, r_k]$ for all $k \le N - 1$. Since $\mu_N = \alpha_N\lambda[l_N, p_N] + (1 - \alpha_N)\lambda[l_N, r_N]$ there exists a real number $\beta$ such that \[
\mu_N = \beta\lambda[l_N, p_N] + (1 - \beta)\lambda[p_N, r_N].\]
Since in addition $\alpha_N \in (0, 1)$ and $p_N < r_N$, we conclude that $0 < \beta < 1$. Consider the following $N$-tuples of probability measures $\vec{\mu}^{\,(1)}$ and $\vec{\mu}^{\,(2)}$:
\begin{align*}
    \mu_k^{(1)} &= \begin{cases}
    \lambda[l_k, r_k] &\text{if $k \le N - 2$},\\
    \lambda[r_{N - 1} - \beta(r_{N - 1} - l_{N - 1}), r_{N - 1}] &\text{if $k = N - 1$},\\
    \lambda[l_N, p_N] &\text{if $k = N$};
    \end{cases}\\
    \mu_k^{(2)} &= \begin{cases}
    \lambda[l_k, r_k] &\text{if $k \le N - 2$},\\
    \lambda[l_{N - 1}, r_{N - 1} - \beta(r_{N - 1} - l_{N - 1})] &\text{if $k = N - 1$},\\
    \lambda[p_N, r_N] &\text{if $k = N$}.
    \end{cases}
\end{align*}

One can easily verify that $\beta\vec{\mu}^{\,(1)} + (1 - \beta)\vec{\mu}^{\,(2)} = \vec{\mu}$. In addition, since $\mu_N^{(1)} = \lambda[l_N, p_N] \ne \mu_N$, we have $\vec{\mu} \ne \vec{\mu}^{\,(1)}$. We claim that both $\vec{\mu}^{\,(1)}$ and $\vec{\mu}^{\,(2)}$ satisfy all the assumptions of \cref{cor:lebesgue_tuple_in_V_C_eq}, and therefore they are both contained in $\mathcal{V}^N[C]$.

First, let us verify that $\vec{\mu}^{\,(1)} \in \mathcal{V}^N[C]$. We have $\mathbb{E}(\mu_k) = (l_k + r_k) / 2$ for all $k \le N - 1$ and \[
\mathbb{E}(\mu_N) = \beta \frac{l_N + p_N}{2} + (1 - \beta)\frac{p_N + r_N}{2} = \frac{(1 - \beta)(r_N - l_N)}{2} + \frac{l_N + p_N}{2};
\]
summarizing this equations we get
\[
C = \mathbb{E}(\mu_1) + \dots + \mathbb{E}(\mu_N) = \frac{l_1 + r_1}{2} + \dots + \frac{l_{N - 1} + r_{N - 1}}{2} + \frac{(1 - \beta)(r_N - l_N)}{2} + \frac{l_N + p_N}{2}.
\]
Substituting this into equation \cref{eq:rn_ln_eq_C_minus_sum_l} we get
\[
r_{N - 1} - l_{N - 1} = r_N - l_N = \frac{r_1 - l_1}{2} + \dots  + \frac{r_N - l_N}{2} + \frac{p_N - l_N}{2} - \frac{\beta(r_N - l_N)}{2}.
\]
In particular,
\begin{equation}\label{eq:final_mu1_satisfy_cor}
\beta(r_{N - 1} - l_{N - 1}) = \beta(r_N - l_N) = (r_1 - l_1) + \dots + (r_{N - 2} - l _{N - 2}) + (p_N - l_N).
\end{equation}

Hence, we only need to verify that $\mathbb{E}(\mu_1^{(1)}) + \dots + \mathbb{E}(\mu_N^{(1)}) = C$. We have $\mathbb{E}(\mu_k^{(1)}) = (l_k + r_k) / 2$ for all $k \le N - 2$ and $\mathbb{E}(\mu_N) = (l_N + p_N) / 2$. In addition, 
\[
\mathbb{E}(\mu_{N - 1}^{(1)}) = \frac{2r_{N - 1} - \beta(r_{N - 1} - l_{N - 1})}{2} =  \frac{l_{N - 1} + r_{N - 1}}{2} + \frac{(1 - \beta)(r_{N - 1} - l_{N - 1})}{2}
\]
and using the equation $r_{N - 1} - l_{N - 1} = r_N - l_N$ we get
\begin{equation}\label{eq:sum_mu1_expectations}
\mathbb{E}(\mu_1^{(1)}) + \dots + \mathbb{E}(\mu_N^{(1)}) = \frac{l_1 + r_1}{2} + \dots + \frac{l_{N - 1} + r_{N - 1}}{2} + \frac{(1 - \beta)(r_N - l_N)}{2} + \frac{l_N + p_N}{2}= C.
\end{equation}
Thus, it follows from equations \cref{eq:final_mu1_satisfy_cor,eq:sum_mu1_expectations} that the $N$-tuple of probability measures $\vec{\mu}^{\,(1)}$ satisfies all the assumptions of \cref{cor:lebesgue_tuple_in_V_C_eq}, and therefore $\vec{\mu}^{\,(1)} \in \mathcal{V}^N[C]$.

Let us verify in the same manner that $\vec{\mu}^{\,(2)} \in \mathcal{V}^N[C]$. Since $\mathbb{E}(\mu_1) + \dots + \mathbb{E}(\mu_N) = C$ and $\mathbb{E}(\mu_1^{(1)}) + \dots + \mathbb{E}(\mu_N^{(1)}) = C$, and since $\vec{\mu} = \beta\vec{\mu}^{\,(1)} + (1 - \beta)\vec{\mu}^{\,(2)}$, it follows from the linearity of the function $\mathbb{E}$ that
\[
\mathbb{E}(\mu_1^{(2)}) + \dots + \mathbb{E}(\mu_N^{(2)}) = C.
\]
Adding $(1 - \beta)(r_{N - 1} - l_{N - 1}) - (p_N - l_N)$ to the both sides of equation \cref{eq:final_mu1_satisfy_cor} and using the equation $r_{N - 1} - l_{N -1} = r_N - l_N$ we also get
\begin{align*}
(r_1 - l_1) + \dots &+ (r_{N - 2} - l_{N - 2}) + (1 -\beta)(r_{N - 1} - l_{N - 1}) \\&= (r_{N - 1} - l_{N - 1}) - (p_N - l_N)= (r_N - l_N) - (p_N - l_N) = r_N - p_N.
\end{align*}
Hence, the $N$-tuple of probability measures $\vec{\mu}^{\,(2)}$ also satisfies all the assumptions of \cref{cor:lebesgue_tuple_in_V_C_eq}, and therefore $\vec{\mu}^{\,(2)} \in \mathcal{V}^N[C]$. Thus, $\vec{\mu}$ is not a subextreme point of $\mathcal{V}^N[C]$.

This contradiction proves that if $\vec{\mu}$ is a subextreme point of $\mathcal{V}^N[C]$, then $\mu_k = \delta(l_k)$ for all $k = 1, \dots, N$. Then $\mathbb{E}(\mu_k) = l_k$, and we also have
\[
C = \mathbb{E}(\mu_1) + \dots + \mathbb{E}(\mu_N) = l_1 + \dots + l_N.
\]

Finally, let us verify that if $\vec{\mu} = (\delta(l_1), \dots, \delta(l_N))$, where $l_1 + \dots + l_N = C$, then $\vec{\mu}$ is a subextreme point of $\mathcal{V}^N[C]$. We have $ \mathbb{E}(\mu_1) + \dots + \mathbb{E}(\mu_N) = C$ and $\mu_k \in \mathcal{D}[l_k, l_k]$ for all $k = 1, \dots, N$. The pair $(\vec{l}, \vec{l})$ is a $C$-compatible boundary, and therefore $\vec{\mu} \in \mathcal{V}^N[C]$. We also have $\mu_k$ is an extreme point of $\mathcal{P}(\R)$, and therefore $\vec{\mu}$ is an extreme point of $\mathcal{P}^N(\R)$. Thus, since $\mathcal{V}^N[C] \subset \mathcal{P}^N(\R)$, the tuple $\vec{\mu}$ is a subextreme point of $\mathcal{V}^N[C]$.
\end{proof}
\section{The inclusion of \texorpdfstring{$V^N[C]$}{VN[C]} into the set of flat \texorpdfstring{$N$}{}-tuples of measures}\label{sec:final}
\subsection{Extreme points of the set of flat \texorpdfstring{$N$}{N}-tuples of probability measures}

\begin{definition}
Let $K$ be a closed segment on the real line, and let $C$ be a real number. We denote by $\mathcal{F}^N(K; C)$ the set of $N$-tuples of probability measures $(\mu_1, \dots, \mu_N) \in \mathcal{P}(K)^N$ such that there exists a transport plan $\gamma \in \mathcal{P}(K^N)$ concentrated on the hyperplane $\{x_1 + \dots + x_N = C\}$ such that $\prj{k}{\gamma} = \mu_k$ for all $k$.
\end{definition}
\begin{proposition} \label{prop:F_KC_is_compact_convex}
The set $\mathcal{F}^N(K; C)$ is a compact convex subset of $\mathcal{M}^N(K)$.
\end{proposition}
\begin{proof}
Consider the set \[
S = K^N \cap \{x_1 + \dots + x_N = C\}.
\]
By definition, the set $\mathcal{F}^N(K; C)$ is the image of $\mathcal{P}(S) \subset \mathcal{M}(K^N)$ under the linear mapping $\mu \mapsto (\prj{1}{\mu}, \dots, \prj{N}{\mu})$. Since $S$ is compact, the set $\mathcal{P}(S)$ is a compact convex subset of $\mathcal{M}(K^N)$, and therefore, since the mapping is continuous, the image of $\mathcal{P}(S)$ is a compact convex subset of $\mathcal{M}^N(K)$. 
\end{proof}

Since $\mathcal{F}^N(K; C)$ is a compact convex subset of the locally convex vector space $\mathcal{M}^N(K)$, by the Krein-Milman theorem this set is fully described by its extreme points

\begin{proposition}\label{prop:extreme_points_of_F}
An $N$-tuple of probability measures $\vec{\mu}$ is an extreme point of the set $\mathcal{F}^N(K; C)$ if and only if there exists an $N$-tuple of points $\{t_k\}_{k = 1}^N \subset K$ such that $t_1 + \dots + t_N = C$ and $\mu_k = \delta(t_k)$ for all $k = 1, \dots, N$.
\end{proposition}
\begin{proof}
Let $\vec{\mu}$ be an extreme point of $\mathcal{F}^N(K; C)$. Suppose that there exists an item $\mu_m$ which is not a Dirac measure. Then there exists a measurable set $A \subset K$ such that $0 < \mu_m(A) < 1$. Denote $\alpha = \mu_m(A)$. Since $\vec{\mu} \in \mathcal{F}^N(K; C)$, there exists a transport plan $\gamma \in \mathcal{P}(K^N)$ concentrated on the hyperplane $\{x_1 + \dots + x_N = C\}$ such that $\prj{k}{\gamma} = \mu_k$ for all $k$.

Let $\gamma^{(1)}$ be the restriction of $\gamma$ to the set $\R^{m - 1} \times A \times \R^{N - m}$, and let $\gamma^{(2)}$ be the restriction of $\gamma$ to the set $\R^{m - 1} \times (\R \backslash A) \times \R^{N - m}$. We trivially have $\gamma = \gamma^{(1)} + \gamma^{(2)}$. Hence, both measures $\gamma^{(1)}$ and $\gamma^{(2)}$ are concentrated on the hyperplane $\{x_1 + \dots + x_N = C\}$. We have $\prj{m}{\gamma^{(1)}} = \left.\mu_m\right\rvert_A$, and therefore $\norm{\gamma^{(1)}} = \norm{\left.\mu_m\right\rvert_A} = \alpha > 0$. Similarly, $\norm{\gamma^{(2)}} = 1 - \alpha$, and therefore both measures $\gamma^{(1)}$ and $\gamma^{(2)}$ are nonzero.

Denote \[
\vec{\mu}^{\,(1)} = \left(\prj{1}{\gamma^{(1)}}, \dots, \prj{N}{\gamma^{(1)}}\right) /  \alpha \text{ and } \vec{\mu}^{\,(2)} = \left(\prj{1}{\gamma^{(2)}}, \dots, \prj{N}{\gamma^{(2)}}\right) /  (1 - \alpha).\]
 Both $N$-tuples of measures $\vec{\mu}^{\,(1)}$ and $\vec{\mu}^{\,(2)}$ are contained in the space $\mathcal{F}^N(K; C)$, and $\vec{\mu} = \alpha \vec{\mu}^{\,(1)} + (1 - \alpha)\vec{\mu}^{\,(2)}$. In addition, since $\alpha\mu_m^{(1)} = \left.\mu_m\right\rvert_A$ and $(1 - \alpha)\mu_m^{(2)} = \left.\mu_m\right\rvert_{\R \backslash A}$, we conclude that $\mu_m^{(1)} \ne \mu_m^{(2)}$, and therefore $\vec{\mu}^{\,(1)} \ne \vec{\mu}^{\,(2)}$. This contradicts the extremality of $\vec{\mu}$.

This contradiction proves that if $\vec{\mu}$ is an extreme point of $\mathcal{F}^N(K;C)$, then all items of $\vec{\mu}$ are Dirac measures; hence, there exists an $N$-tuple of points $\{t_k\}_{k = 1}^N \subset K$ such that $\mu_k = \delta(t_k)$ for all $k$. Let us verify that if $(\delta(t_1), \dots, \delta(t_N))$ is contained in $\mathcal{F}^N(K; C)$, then $t_1 + \dots + t_N = C$. Indeed, let $\gamma$ be a transport plan concentrated on the hyperplane $\{x_1 + \dots + x_N = C\}$ such that $\prj{k}{\gamma} = \delta(t_k)$ for all $k$. Since $\supp(\prj{k}{\gamma}) = \{t_k\}$, we conclude that $\supp(\gamma) \subset R^{k - 1} \times \{t_k\} \times \R^{N - k}$. Intersecting that sets for all $k$, we obtain $\supp(\gamma) = \{(t_1, \dots, t_N)\}$; hence, $\gamma = \delta(t_1, \dots, t_N)$ and therefore $t_1 + \dots + t_N = C$.

Finally, let us verify that if $(t_1, \dots, t_N)$ is an $N$-tuple of points such that $t_k \in K$ for all $k$ and $t_1 + \dots + t_N = C$, then an $N$-tuple of measures $(\delta(t_1), \dots, \delta(t_N))$ is an extreme point of $\mathcal{F}^N(K; C)$. The transport plan $\gamma = \delta(t_1, \dots, t_N) \in \mathcal{P}(K^N)$ is concentrated on the hyperplane $\{x_1 + \dots + x_N = C\}$ and $\prj{k}{\gamma} = \delta(t_k)$ for all $k$; hence, $(\delta(t_1), \dots, \delta(t_N)) \in \mathcal{F}^N(K; C)$. The measure $\delta(t_k)$ is an extreme point of $\mathcal{P}(K)$ for all $k = 1, \dots, N$; hence, $(\delta(t_1), \dots, \delta(t_N))$ is an extreme point of $\mathcal{P}^N(K)$, and therefore, since $\mathcal{F}^N(K; C) \subset \mathcal{P}^N(K)$, we conclude that $(\delta(t_1), \dots, \delta(t_N))$ is also an extreme point of $\mathcal{F}^N(K; C)$.
\end{proof}

\subsection{Proof of the main theorem} The set $\mathcal{V}^N[C]$ is noncompact. In what follows, we consider the intersection of this set with the compact set $\mathcal{P}^N(K)$ and describe the closed convex hull of the intersection.
\begin{definition}
Given a closed segment $K$ and a real number $C$. Denote \[\mathcal{V}^N(K; C) = \mathcal{V}^N[C] \cap \mathcal{P}^N(K).\]
\end{definition}

\begin{proposition}\label{prop:U_KC_is_closed}
The set $\mathcal{V}^N(K; C)$ is a compact subset of $\mathcal{P}^N(K)$. 
\end{proposition}
\begin{proof}
Since $\mathcal{P}^N(K)$ is compact, it is enough to prove that $\mathcal{V}^N(K; C)$ is closed. Let $\{\vec{\mu}^{\,(n)}\}_{n = 1}^{\infty} \subset \mathcal{V}^N(K; C)$ be a sequence of $N$-tuples of measures converging to $\vec{\mu}$. We claim that $\vec{\mu} \in \mathcal{V}^N(K; C)$. 

The sequence $\{\mu_k^{(n)}\}_{n = 1}^{\infty}$ converges weakly to $\mu_k$ for each $k = 1, \dots, N$. Hence, \[\lim_{n \to \infty}\mathbb{E}\left(\mu_k^{(n)}\right) = \mathbb{E}(\mu_k) \text{ for each $k = 1, \dots, N$},\]
and therefore $\mathbb{E}(\mu_1) + \dots + \mathbb{E}(\mu_N) = C$. By the definition of $\mathcal{V}^N(K; C)$, for each $n$ there exists a $C$-compatible boundary $(\vec{l}^{\,(n)}, \vec{r}^{\,(n)})$ such that $\mu_k^{(n)} \in \mathcal{D}[l_k^{(n)}, r_k^{(n)}]$ for each $k$.

Since $\mu_k^{(n)} \in \mathcal{D}[l_k^{(n)}, r_k^{(n)}]$, there exists a point $p_k^{(n)}$ such that $l_k^{(n)} \le p_k^{(n)} \le r_k^{(n)}$, $\mu_k^{(n)} \in \mathcal{D}[l_k^{(n)}, p_k^{(n)}]$ and $\supp(\mu_k^{(n)}) = [l_k^{(n)}, p^{(n)}_k] \subset K$. Hence, it follows from \cref{prop:strong_closure_of_D} that $\mu_k \in \mathcal{D}[l_k, p_k]$ for each $k$, where $l_k = \lim_{n \to \infty}l_k^{(n)}$ and $p_k = \liminf_{n \to \infty}p_k^{(n)}$.

Let $\vec{l} = (l_1, \dots, l_N)$, and let $\vec{r} = (r_1, \dots, r_N)$. Let us verify that $(\vec{l}, \vec{r})$ is a $C$-compatible boundary. Since $(\vec{l}^{\,(n)}, \vec{r}^{\,(n)})$ is a $C$-compatible boundary and $p_k^{(n)} \le r_k^{(n)}$ for all $k$ and for all $n$, we have  
\[p_k^{(n)} \le r_k^{(n)} \le C - \left(l_1^{(n)} + \dots + l_N^{(n)}\right) + l_k^{(n)}.\]
Hence,
\[
p_k = \liminf_{n \to \infty}p_k^{(n)} \le \lim_{n \to \infty}\left(C - \left(l_1^{(n)} + \dots + l_N^{(n)}\right) + l_k^{(n)}\right) = C - (l_1 + \dots + l_N) + l_k,
\]
and therefore $p_k - l_k \le C - (l_1 + \dots + l_N)$ for all $k$. In addition, since $p_k^{(n)} \ge  l_k^{(n)}$, we conclude that
\[
p_k = \liminf_{n \to \infty}p_k^{(n)} \ge \lim_{n \to \infty}l_k^{(n)} = l_k.
\]

Hence, the pair $(\vec{l}, \vec{p})$ is a $C$-compatible boundary, and therefore $\vec{\mu} \in \mathcal{V}^N(C)$. Since in addition the set $\mathcal{P}^N(K)$ is closed and $\vec{\mu}^{\,(n)} \in \mathcal{P}^N(K)$ for all $n$, we conclude that $\vec{\mu}$ is also contained in $\mathcal{P}^N(K)$. In particular, $\vec{\mu} \in \mathcal{V}^N[C] \cap \mathcal{P}^N(K) = \mathcal{V}^N(K; C)$. Thus, the set $\mathcal{V}^N(K; C)$ is closed.
\end{proof}

Next, we prove that for every $N$-tuple $\vec{\mu} \in \mathcal{V}^N(K; C)$ there exists a transport plan $\gamma$ concentrated on the hyperplane $\{x_1 + \dots + x_N = C\}$ such that $\prj{k}{\gamma_k} = \mu_k$ for all $k$. More precisely, we prove the following theorem.
\begin{theorem}\label{thm:F_eq_co_V_KC}
The set $\overline{\mathrm{co}}\,\mathcal{V}^N(K; C)$ coincides with $\mathcal{F}^N(K; C)$.
\end{theorem}
\begin{proof}
The set $\overline{\mathrm{co}}\,\mathcal{V}^N(K; C)$ is a compact convex subset of $
\mathcal{M}^N(K)$, and therefore by the Krein-Milman theorem the set $\overline{\mathrm{co}}\,\mathcal{V}^N(K; C)$ coincides with the closed convex hull of its extreme points. Take into account the following theorem:
\begin{theorem*}[Milman converse, {{\cite[Proposition 1.5]{Phelps2001}}}]\label{thm:Milman_converse}
Suppose that $X$ is a compact convex subset of a locally convex space, that $Z \subset X$, and that $X$ is the
closed convex hull of $Z$. Then the extreme points of X are contained in the closure of $Z$.
\end{theorem*}

\noindent By \cref{prop:U_KC_is_closed} the set $\mathcal{V}^N(K; C)$ is a closed subset of $\overline{\mathrm{co}}\,\mathcal{V}^N(K; C)$. Thus, by the Milman converse the set $\mathcal{V}^N(K; C)$ contains all extreme points of $\overline{\mathrm{co}}\,\mathcal{V}^N(K; C)$.

If $\vec{\mu}$ is an extreme point of $\overline{\mathrm{co}}\,\mathcal{V}^N(K; C)$ and $\vec{\mu} \in \mathcal{V}^N(K; C)$, then one can trivially verify that $\vec{\mu}$ is a subextreme point of $\mathcal{V}^N(K; C)$. So, let us find all subextreme points of $\mathcal{V}^N(K; C)$. 

We have $\mathcal{V}^N(K; C) \subset \mathcal{V}^N[C]$. Hence, if $\vec{\mu} \in \mathcal{V}^N(K; C)$ and $\vec{\mu}$ is a subextreme point of $\mathcal{V}^N[C]$, then $\vec{\mu}$ is also a subextreme point of $\mathcal{V}^N(K; C)$. On the other hand, let $\vec{\mu}$ be a subextreme point of $\mathcal{V}(K; C)$. Suppose that $\vec{\mu} = \alpha \vec{\mu}^{\,(1)} + (1 - \alpha)\vec{\mu}^{\,(2)}$ for some $\alpha \in (0, 1)$ and for some distinct $N$-tuples of measures $\vec{\mu}^{\,(1)}, \vec{\mu}^{\,(2)} \in \mathcal{V}^N[C]$. Since $\supp(\mu_k) \subset K$ and $\mu_k = \alpha \mu_k^{(1)} + (1 - \alpha)\mu_k^{(2)}$, we conclude that $\supp(\mu_k^{(1)}) \subset K$ and $\supp(\mu_k^{(2)}) \subset K$ for all $k$. Hence, both $N$-tuples $\vec{\mu}^{\,(1)}$ and $\vec{\mu}^{\,(2)}$ are contained in $\mathcal{V}^N(K; C)$, and this contradicts the fact that $\mu$ is a subextreme point of $\mathcal{V}^N(K; C)$.

Thus, $\vec{\mu}$ is a subextreme point of $\mathcal{V}^N(K; C)$ if and only if $\vec{\mu} \in \mathcal{V}^N(K; C)$ and $\vec{\mu}$ is a subextreme point of $\mathcal{V}^N[C]$. In particular, using \cref{thm:2_ex_of_V_C} we get
\[
\mathrm{se}\,\mathcal{V}^N(K; C) = \{(\delta(l_1), \dots, \delta(l_N)) \colon l_1 + \dots + l_N = C\text{ and } \{l_1, \dots, l_N\} \subset K\}.
\]

Since the set $\mathrm{se}\,\mathcal{V}^N(K; C)$ contains all extreme points of $\overline{\mathrm{co}}\,\mathcal{V}^N(K; C)$, we get $\overline{\mathrm{co}}\,\mathcal{V}^N(K; C) = \overline{\mathrm{co}}\,\mathrm{se}\,\mathcal{V}^N(K; C)$. In addition, it follows from \cref{prop:extreme_points_of_F} that $\mathrm{ex}\,\mathcal{F}(K; C) = \mathrm{se}\,\mathcal{V}^N(K; C)$, and therefore, since $\mathcal{F}^N(K; C)$ is a compact convex set, we get \[
\mathcal{F}^N(K; C) = \overline{\mathrm{co}}\,\mathrm{ex}\,\mathcal{F}(K; C) = \overline{\mathrm{co}}\,\mathrm{se}\,\mathcal{V}^N(K; C) = \overline{\mathrm{co}}\,\mathcal{V}^N(K; C).\]
\end{proof}

Now we can prove \cref{thm:main_theorem}. Let us recall its formulation.
\begin{theorem*}
Let $\{\mu_k\}_{k = 1}^N$ be absolutely continuous probability measures on the real line. Suppose that $\supp(\mu_k) = [l_k, r_k]$ and the density function of $\mu_k$ is nonincreasing on the segment $[l_k, r_k]$ for all $k = 1, \dots, N$. Then the $N$-tuple $\{\mu_k\}_{ = 1}^n$ is flat if and only if we have $r_k - l_k \le \mathbb{E}(\mu_1) + \dots + \mathbb{E}(\mu_N) - (l_1 + \dots + l_N)$ for all $k = 1, \dots, N$.
\end{theorem*}\begin{proof}
First, let us prove the necessity. Let $\gamma$ be a transport plan concentrated on the hyperplane $\{x_1 + \dots + x_N = C\}$ such that $\prj{k}{\gamma} = \mu_k$ for all $k = 1, \dots, N$. We have
\begin{align*}
C = \int_{\R^N}(x_1 + \dots + x_N)\,\gamma(dx_1, \dots, dx_N) &= \int_{l_1}^{r_1}x_1\,\mu_1(dx_1) + \dots + \int_{l_N}^{r_N}x_N\,\mu_N(dx_N) \\
&= \mathbb{E}(\mu_1) + \dots + \mathbb{E}(\mu_N),
\end{align*}
and therefore the constant $C$ is uniquely defined.

Consider the set \[S = \left([l_1, r_1] \times \dots \times [l_N, r_N]\right) \cap \{x_1 + \dots + x_N = C\}.\]
Since $\supp(\prj{k}{\gamma}) = [l_k, r_k]$ for each $k$, we conclude that $\supp(\gamma) \subset [l_1, r_1] \times \dots \times [l_N, r_N]$. Since in addition $\supp(\gamma) \subset \{x_1 + \dots + x_N = C\}$, we get $\supp(
\gamma) \subset S$.

If $(x_1, \dots, x_k) \in S$, then for each $k$ the following inequality holds:
\[
x_k = C - \sum_{i \ne k}x_i \le C - \sum_{i \ne k}l_i = C + l_k - (l_1 + \dots + l_N).
\]
Hence, we get $[l_k, r_k] = \supp(\prj{k}{\gamma}) \subset (-\infty, C + l_k - (l_1 + \dots + l_N)]$, and therefore 
\[
r_k - l_k \le C - (l_1 + \dots + l_N) = \mathbb{E}(\mu_1) + \dots + \mathbb{E}(\mu_N) - (l_1 + \dots + l_N)
\]
for each $k = 1, \dots, N$.

Let us prove the sufficiency. Denote $C = \mathbb{E}(\mu_1) + \dots + \mathbb{E}(\mu_N)$. It follows from the assumptions that the pair $(\vec{l}, \vec{r})$ is a $C$-compatible boundary. In addition, $\mu_k \in \mathcal{D}_{AC}[l_k, r_k] \subset \mathcal{D}[l_k, r_k]$ for each $k$. Thus, the $N$-tuple of measures $\vec{\mu} = (\mu_1, \dots, \mu_N)$ is contained in $\mathcal{V}^N[C]$.

There exists a (large enough) closed segment $K$ such that $[l_k, r_k] \subset K$ for each $k$. Then $\mu_k \in \mathcal{P}(K)$ for each $k$, and therefore $\vec{\mu} \in \mathcal{V}^N(K; C)$. It follows from \cref{thm:F_eq_co_V_KC} that $\mathcal{V}^N(K; C) \subset \mathcal{F}^N(K; C)$, and so by the definition of $\mathcal{F}(K; C)$ there exists a transport plan $\gamma$ concentrated on the hyperplane $\{x_1 + \dots + x_N = C\}$ such that $\prj{k}{\gamma} = \mu_k$ for all $k$. Thus, $\vec{\mu}$ is a flat $N$-tuple of measures.
\end{proof}
\printbibliography

\end{document}